\documentclass[a4paper,11pt]{article}

\usepackage{amsthm,amsmath,amssymb}
\usepackage[latin1]{inputenc}
\usepackage[english]{babel}
\usepackage{mathscinet}
\usepackage{graphics}
\usepackage{graphicx}
\usepackage[all,cmtip]{xy}
\usepackage{multirow}
\usepackage{hyperref}

\theoremstyle{plain}
\newtheorem{theo}{Theorem}[section]
\newtheorem{prop}[theo]{Proposition}
\newtheorem{cor}[theo]{Corollary}
\newtheorem{lemma}[theo]{Lemma}

\theoremstyle{definition}
\newtheorem{defi}[theo]{Definition}
\newtheorem{ex}[theo]{Example}

\theoremstyle{remark}
\newtheorem{remark}[theo]{Remark}
\newtheorem{nothing}[theo]{\noindent\!\!\bf}

\DeclareMathOperator{\lcm}{lcm}
\DeclareMathOperator{\Sing}{Sing}

\DeclareMathOperator{\ord}{ord}

\DeclareMathOperator{\pr}{pr}

\def\Z{\mathbb{Z}}
\def\N{\mathbb{N}}
\def\C{\mathbb{C}}
\def\Q{\mathbf{Q}}

\def\P{\mathbb{P}}
\def\w{\omega}
\def\l{\ell}
\def\E{\mathcal{E}}

\newcommand\bd{\mathbf{d}}

\newcommand\bxi{\boldsymbol{\xi}}
\newcommand\be{\mathbf{e}}
\newcommand\br{\mathbf{r}}
\newcommand\bs{\mathbf{s}}
\newcommand\bt{\mathbf{t}}

\title{\bf Embedded $\mathbf{Q}$-Resolutions for Yomdin-L\^{e} Surface Singularities}
\author{Jorge Mart\'{\i}n-Morales\footnote{Partially supported by the projects MTM2010-21740-C02-02, ``E15 Grupo Consolidado Geometr\'{\i}a'' from the goverment of Aragón, FQM-333 from ``Junta de Andalucía'', and PRI-AIBDE-2011-0986 Acción Integrada hispano-alemana.}}
\date{Centro Universitario de la Defensa - IUMA.\\ Academia General Militar, Ctra.~de Huesca s/n.\\ 50090, Zaragoza, Spain.\\ jorge@unizar.es}

\begin{document}

\maketitle

\vspace{-0.25cm}

\begin{abstract}
In a previous work we have introduced and studied the notion of embedded $\mathbf{Q}$-resolution, which essentially consists in allowing the final ambient space to contain abelian quotient singularities. Here we explicitly compute an embedded $\mathbf{Q}$-resolution of a Yomdin-Lê surface singularity $(V,0)$ in terms of a (global) embedded $\mathbf{Q}$-resolution of their tangent cone by means of just weighted blow-ups at points. The generalized A'Campo's formula in this setting is applied so as to compute the characteristic polynomial. As a consequence, an exceptional divisor in the resolution of $(V, 0)$, apart from the first one which might be special, contributes to its complex monodromy if and only if so does the corresponding divisor in the tangent cone. Thus the resolution obtained is optimal in the sense that the weights can be chosen so that every exceptional divisor in the $\mathbf{Q}$-resolution of $(V, 0)$, except perhaps the first one, contributes to its monodromy.

\vspace{0.25cm}


\vspace{0.25cm}

\end{abstract}


\section*{Introduction}

Let $(V,0)\subset (\C^3,0)$ be a germ of surface singularity in $\C^3$. By definition, $V$ is the zero set of a holomorphic function $f:U \to \C$, where $U\subset \C^3$ is a small neighborhood of the origin and $f(0)=0$. Denote also by $f$ the germ at the origin of this function; it is an element of the local ring $\C\{x,y,z\}$.

Consider the decomposition of $f$ into homogeneous parts,
$$
  f(x,y,z) = f_m(x,y,z) + f_{m+1}(x,y,z) + \cdots,
$$
that is, $f_i$ is homogeneous of degree $i$ and $f_m\neq 0$. The integer $m$ is the multiplicity of the singularity and the order of the series $f$. Denote by ${\bf C}:= V(f_m)\subset \P^2$ the projective plane curve defined by the tangent cone of the singularity. The following two families are considered in this work separately:
\begin{enumerate}
\item {\em Superisolated singularity} (or, shortly, SIS): the local equation $f$ satisfies
$\Sing({\bf C})\cap V(f_{m+1}) = \emptyset$ as a subset in $\P^2$.
\item {\em Yomdin-L\^{e} singularity} (YLS): the decomposition of $f$ into homogeneous polynomials is of the form $f=f_m + f_{m+k} + \cdots$, $k\geq 1$, and the condition $\Sing({\bf C})\cap V(f_{m+k}) = \emptyset$ holds in $\P^2$.
\end{enumerate}

These singularities have been extensively studied by many authors, see for instance the survey \cite{ALM06} where part of the theory of these singularities and their applications including some new and recent developments are reviewed. The SIS, i.e.~$k=1$, were introduced by Luengo and also appear in a paper by Stevens, where the $\mu$-constant stratum is considered, see \cite{Luengo87} and \cite{Stevens89}. Afterward Artal described in his PhD thesis \cite{Artal94} an embedded resolution of such singularities using blow-ups at points and rational curves. However, no embedded resolution is found in the literature for YLS with $k\geq 2$.

\vspace{0.25cm}

In this paper, the new techniques developed in~\cite{AMO11a, AMO11b, Martin11} are partially applied to study these two families of singularities. More precisely, we present here a detailed explicit description of an embedded $\Q$-resolution for YLS in terms of a (global) embedded $\Q$-resolution of their tangent cone. It is proven that only weighted blow-ups at points are needed. By contrast, the final total space produced has abelian quotient singularities.

\vspace{0.25cm}

The main result of this paper is a collection of several results that can be summarized as follows, cf.~Lemma~\ref{step1}, Proposition~\ref{step_b_SIS}, Theorem~\ref{summery_theorem} for SIS and Lemma~\ref{step1_YS}, Proposition~\ref{step_b_YS}, Theorem~\ref{summery_theorem_YS} for YLS.

\begin{theo}
Let $\varrho^{P}: Y^{P} \to (\C^2, P)$ be an embedded $\Q$-resolution of the tangent cone $({\bf C},P)$ for each $P \in \Sing({\bf C})$. Assume that
$$
(\varrho^{P})^{*}({\bf C},P) = \widehat{{\bf C}} + \sum_{a \in S(\Gamma_{+}^{P})} m_a^{P} \mathcal{E}_a^{P}
$$
is the total transform of $({\bf C},P)$, where $\E_a^{P}$ is the exceptional divisor of the $(p_a^P,q_a^P)$-blow-up at a point $P_a$ belonging to the locus of non-transversality. Denote by $\nu_a^P$ the $(p_a^P,q_a^P)$-multiplicity of ${\bf C}$ at $P_a$.

Then, one can construct an embedded $\Q$-resolution $\rho: X \to (\C^3,0)$ of the Yomdin-L\^{e} singularity $(V,0)$ such that the total transform is
$$
  \rho^{*}(V,0) = \widehat{V} + m \, E_0 + \sum_{\begin{subarray}{c} P \in \Sing({\bf C}) \\ a \in S(\Gamma_{+}^{P})\end{subarray}} \frac{(m+k) \cdot m_a^P}{\gcd(k,m_a^P)} \ E_{a}^{P},
$$
and $E_a^P$ appears after the $\left(\frac{k \, p_a^P}{\gcd(k,\nu_a^P)}, \frac{k \, q_a^P}{\gcd(k, \nu_a^P)}, \frac{\nu_a^P}{\gcd(k,\nu_a^P)} \right)$-blow-up at the point $P_a$ (the locus of non-transversality in dimension 2 and 3 are identified).
\end{theo}

For $k=1$, the main advantage compared with Artal's resolution~\cite{Artal94} is that in the latter $\nu_a^P$ (rather than just one) blow-ups at points and rational curves at each step are needed to achieve a similar situation. On the other hand, as it is said above, no embedded resolution for YLS with $k\geq 2$ can be found in the literature. The main difficulty in computing a (usual) embedded resolution of this kind of singularities is that after several blow-ups at points and rational curves, following the ideas of~\cite{Artal94}, one eventually obtains a branch of resolutions depending on $k$. Thus the study of this singularities by using the classical tools does not seem to be very helpful. 

\vspace{0.15cm}

The generalized A'Campo's formula \cite[Theorem 2.8]{Martin11} is applied and the characteristic polynomial and the Milnor number are calculated as an application, see Theorem~\ref{char_poly_SIS} and Corollary~\ref{cor_char_poly_SIS} for SIS and Theorem~\ref{char_poly_YS} and Corollary~\ref{cor_char_poly_YS} for YLS. In particular, the formulas by D.~Siersma~\cite{Siersma90} and J.~Stevens~\cite{Stevens89} for the characteristic polynomial of YLS can be obtained in this way. Other more sophisticated invariants, including mixed Hodge structure of the cohomology of the Milnor fiber, are the subjects of our study for the future.

\vspace{0.15cm}

As a consequence, we show that an exceptional divisor $E_a^P$ in the resolution of $(V, 0)$ contributes to the complex monodromy if and only if so does the corresponding divisor $\E_a^P$ in the tangent cone, see Lemmas~\ref{euler_char_SIS} and~\ref{euler_char_YS}. Thus the weights can be chosen so that every exceptional divisor in the $\Q$-resolution of $(V, 0)$, except perhaps the first one $E_0$, contributes to its monodromy.

\vspace{0.15cm}

Although the proofs presented here are a bit technical, which involve a lot of calculations with local equations on charts, the final construction is very useful. In fact, this work can be considered as the first step in the computation of the mixed Hodge structures together with the monodromy action of YLS. Note that, following the ideas of~\cite{Artal94}, these tools can be used in combination with the generalized Steenbrink's spectral sequence of~\cite{Martin11b} to find two YLS having the same characteristic polynomials, the same abstract topologies, but different embedded topologies (it is enough to take a Zariski pairs in the tangent cones). Besides, these techniques can be applied to study superisolated singularities in higher dimension, see~\cite[\S\,VI.4]{Martin11PhD}, and the same applies to weighted Yomdin-L\^{e} surface singularities, see~\cite[\S\,VII.3]{Martin11PhD}.


Although these two families can be studied simultaneously, for better exposition they are presented and treated separately. The paper is organized as follows. In \S\ref{sec_prelim}, some well-known preliminaries about weighted blow-ups and embedded $\Q$-resolutions are presented. After recalling the step zero in Artal's resolution in \S\ref{sec_preparations_SIS}, the full construction of the embedded $\Q$-resolution for SIS is given in \S\ref{complete_description_SIS} so as to prove the main theorem for this family. In \S\ref{sec_char_poly_SIS}, the Euler characteristic of the strata needed for applying A'Campo's formula is calculated and the characteristic polynomial and the Milnor number are obtained as an application. Finally, \S\ref{sec_preparations_YS}, \S\ref{complete_description_YS}, \S\ref{sec_char_poly_YS} are the analogous of \S\ref{sec_preparations_SIS}, \S\ref{complete_description_SIS}, \S\ref{sec_char_poly_SIS} for YLS showing the corresponding results mentioned above.


\begin{center}
\begin{large}
\textbf{Acknowledgments}
\end{large}
\end{center}

This is part of my PhD thesis. I am deeply grateful to my advisors Enrique Artal and José Ignacio Cogolludo for supporting me continuously with their fruitful conversations and ideas.

\section{Preliminaries}\label{sec_prelim}

Let us sketch some definitions and properties about $V$-manifolds, weighted projective spaces, and weighted blow-ups, see \cite{AMO11a, AMO11b, Martin11PhD} for a more detailed exposition. Also, the generalized A'Campo's formula for embedded $\Q$-resolutions is recalled, see \cite{Martin11}.

\subsection{Embedded Q-resolutions and weighted blow-ups}
Classically an embedded resolution of $\{f=0\} \subset \C^{n+1}$ is a proper analytic map $\pi: X \to (\C^{n+1},0)$ from a smooth variety $X$ satisfying, among other conditions, that $\pi^{*}(\{f=0\})$ is a normal crossing divisor. To weaken the condition on the preimage of the singularity one studies the following notion.

\begin{defi}
Let $H=\{f=0\}\subset \C^{n+1}$. An {\em embedded $\Q$-resolution} of $(H,0) \subset (\C^{n+1},0)$ is a proper analytic map $\pi: X \to (\C^{n+1},0)$ such that:
\begin{enumerate}
\item $X$ is a $V$-manifold with abelian quotient singularities.
\item $\pi$ is an isomorphism over $X\setminus \pi^{-1}(\Sing(H))$.
\item $\pi^{*}(H)$ is a hypersurface with $\mathbb{Q}$-normal crossings on $X$.
\end{enumerate}
\end{defi}

To deal with these resolutions, some notation needs to be introduced. Let $G := \mu_{d_0} \times \cdots \times \mu_{d_r}$ be an arbitrary finite abelian group written as a product of finite cyclic groups, that is, $\mu_{d_i}$ is the cyclic group of $d_i$-th roots of unity. Consider a matrix of weight vectors
$$
A := (a_{ij})_{i,j} = [{\bf a}_0 \, | \, \cdots \, | \, {\bf a}_n ] \in Mat ((r+1) \times (n+1), \Z)
$$
and the action
\begin{equation*}
\begin{array}{c}
( \mu_{d_0} \times \cdots \times \mu_{d_r} ) \times \C^{n+1}  \longrightarrow  \C^{n+1}, \\[0.15cm]
\big( \bxi_{\bd} , {\bf x} \big)  \mapsto  (\xi_{d_0}^{a_{00}} \cdots \xi_{d_r}^{a_{r0}}\, x_0,\, \ldots\, , \xi_{d_0}^{a_{0n}} \cdots \xi_{d_r}^{a_{rn}}\, x_n ).
\end{array}
\end{equation*}
The set of all orbits $\C^{n+1} / G$ is called ({\em cyclic}) {\em quotient space of type $({\bf d};A)$} and it is denoted by
$$
  X({\bf d}; A) := X \left( \begin{array}{c|ccc} d_0 & a_{00} & \cdots & a_{0n}\\ \vdots & \vdots & \ddots & \vdots \\ d_r & a_{r0} & \cdots & a_{rn} \end{array} \right).
$$

The orbit of an element $(x_0,\ldots,x_n)$ under this action is denoted by $[(x_0,\ldots,x_n)]$. Condition 3 of the previous definition means the total transform $\pi^{-1}(H) = (f\circ \pi)^{-1}(0)$ is locally given by a function of the form $x_0^{m_0} \cdots x_k^{m_k} : X({\bf d};A) \rightarrow \C$, see~\cite{Steenbrink77}. The previous numbers $m_{i}$'s have no intrinsic meaning unless $\mu_{{\bf d}}$ induces a small action on $GL(n+1,\C)$. This motivates the following.

\begin{defi}\label{def_normalized_XdA_intro}
The type $({\bf d}; A)$ is said to be {\em normalized} if the action is free on $(\C^{*})^{n+1}$ and $\mu_{\bf d}$ is identified with a small subgroup of $GL(n+1,\C)$.
\end{defi}

As a tool for finding embedded $\Q$-resolutions one uses weighted blow-ups with smooth center. Special attention is paid to the case of dimension 2 and 3 and blow-ups at points.

\begin{ex}\label{blowup_dim2}
Assume $(d;a,b)$ is normalized and $\gcd (\w) =1$, $\w := (p,q)$. Then, the total space of the $\w$-blow-up at the origin of $X(d;a,b)$,
\begin{equation}\label{w-blow-up_intro}
\pi_{(d;a,b),\w}: \widehat{X(d;a,b)}_{\w} \longrightarrow X(d;a,b),
\end{equation}
can be written as
$$
\widehat{U}_1 \cup \widehat{U}_2 = X \left( \frac{pd}{e}; 1, \frac{-q+\beta p b}{e} \right) \cup X \left( \frac{qd}{e}; \frac{-p+\mu qa}{e}, 1 \right)
$$
and the charts are given by
\begin{equation*}
\begin{array}{c|c}
\text{First chart} & X \left( \displaystyle\frac{pd}{e}; 1, \frac{-q+\beta p b}{e} \right)  \ \longrightarrow \ \widehat{U}_1, \\[0.5cm] & \,\big[ (x^e,y) \big] \mapsto \big[ ((x^p,x^q y),[1:y]_{\w}) \big]_{(d;a,b)}. \\ \multicolumn{2}{c}{} \\
\text{Second chart} & X \left( \displaystyle\frac{qd}{e}; \frac{-p+\mu qa}{e}, 1 \right) \ \longrightarrow \ \widehat{U}_2, \\[0.5cm] & \hspace{0.15cm} \big[ (x,y^e) \big] \mapsto \big[ ((x y^p, y^q),[x:1]_{\w}) \big]_{(d;a,b)}.
\end{array}
\end{equation*}
Above, $e=\gcd(d,pb-qa)$ and $\beta a \equiv \mu b \equiv 1$ $(\text{mod $d$})$. Observe that the origins of the two charts are cyclic quotient singularities; they are located at the exceptional divisor $E$ which is isomorphic to $\P^1_{\w} \cong \P^1$.
\end{ex}

\begin{ex}\label{blowup_dim3_smooth}
Let $\pi_{\w}: \widehat{\C}^3_{\w} \to \C^3$ be the $\w$-weighted blow-up at the origin with $\w=(p,q,r)$, $\gcd(\w)=1$. The new space is covered by three open sets 
$$
\widehat{\C}^3_{\w} = U_1 \cup U_2 \cup U_3 = X(p;-1,q,r) \cup X(q;p,-1,r) \cup X(r;p,q,-1),
$$
and the charts are given by
\begin{equation}\label{charts_dim3}
\begin{array}{cc}
X(p;-1,q,r) \longrightarrow U_1: & [(x,y,z)] \mapsto ((x^p, x^q y, x^r z),[1:y:z]_{\w}), \\[0.25cm]
X(q;p,-1,r) \longrightarrow U_2: & [(x,y,z)] \mapsto ((x y^p,y^q,y^r z),[x:1:z]_{\w}), \\[0.25cm]
X(r;p,q,-1) \longrightarrow U_3: & [(x,y,z)] \mapsto ((x z^p, y z^q, z^r),[x:y:1]_{\w}).
\end{array}
\end{equation}

In general $\widehat{\C}^3_{\w}$ has three lines of (cyclic quotient) singular points located at the three axes of the exceptional divisor $\pi^{-1}_{\w}(0) \simeq \P^2_{\w}$. For instance, a generic point in $x=0$ is a cyclic point of type $\C\times X(\gcd(q,r);p,-1)$.
Note that although the quotient spaces are represented by normalized types, the exceptional divisor can still be simplified:
\begin{equation}\label{propPw}
\begin{array}{rcl}
\P^2(p,q,r) & \longrightarrow & \P^2 \displaystyle\left(\frac{p}{(p,r)\cdot
(p,q)},\frac{q}{(q,p)\cdot (q,r)},
\frac{r}{(r,p)\cdot (r,q)}\right),\\[0.5cm]
\displaystyle \,[x:y:z] & \mapsto & [x^{\gcd(q,r)}:y^{\gcd(p,r)}:z^{\gcd(p,q)}].
\end{array} 
\end{equation}

However, this simplification may be not useful when working with the whole ambient space because its charts are not compatible with $\widehat{\C}^3_{\w}$. Thus the natural covering of the exceptional divisor is
$$
  \P^2_{\w} = V_1 \cup V_2 \cup V_3 = X(p;q,r) \cup X(q;p,r) \cup X(r;p,q),
$$
and the charts are given by the restrictions of the maps in~\eqref{charts_dim3} to $x=0$, $y=0$, and $z=0$ respectively.
\end{ex}

\begin{ex}\label{blowup_dim3_singular}
Assume $(d;a,b,c)$ is normalized and $\gcd (\w) =1$, $\w := (p,q,r)$. Then, the total space of the $\w$-blow-up at the origin of $X(d;a,b,c)$,
$$
\pi = \pi_{(d;a,b,c),\w}:\, \widehat{X(d;a,b,c)}_{\w} \longrightarrow X(d;a,b,c)
$$
can be covered by three open sets as
$$
\widehat{X(d;a,b,c)}_{\w} = \frac{\widehat{\C}^3_{\w}}{\mu_d} = \frac{U_1 \cup U_2 \cup U_3}{\mu_d} = \widehat{U}_1 \cup \widehat{U}_2 \cup \widehat{U}_3,
$$
where

$$
\begin{array}{ccc}
\displaystyle \widehat{U}_1 = \frac{U_1}{\mu_d} = \frac{X(p;-1,q,r)}{\mu_d} =
X \left(\begin{array}{c|ccc} p & -1 & q & r \\ pd & a & pb-qa & pc-ra \end{array}\right), \\[0.75cm]
\displaystyle \widehat{U}_2 = \frac{U_2}{\mu_d} = \frac{X(q;p,-1,r)}{\mu_d} =
X \left(\begin{array}{c|ccc} q & p & -1 & r \\ qd & qa-pb & b & qc-rb \end{array}\right), \\[0.75cm]
\displaystyle \widehat{U}_3 = \frac{U_3}{\mu_d} = \frac{X(r;p,q,-1)}{\mu_d} =
X \left(\begin{array}{c|ccc} r & p & q & -1 \\ rd & ra-pc & rb-qc & c \end{array}\right). 
\end{array}
$$

The charts are given by the induced maps on the corresponding quotient spaces, see Equation~(\ref{charts_dim3}). The exceptional divisor $E = \pi^{-1}_{(d;a,b,c),\w}(0)$ is identified with the quotient
$$
\P^2_{\w}(d;a,b,c) := \frac{\P^2_{\w}}{\mu_d}.
$$
There are three lines of quotient singular points in $E$ and outside $E$ the map $\pi_{(d;a,b,c),\w}$ is an isomorphism.


The expression of the quotient spaces can be modified as follows. Let $\alpha$ and $\beta$ be two integers such that $\alpha d + \beta a = \gcd(d,a)$, then one has that the space $X\left(\begin{smallmatrix} p ; & -1 & q & r \\ pd ; & a & pb-qa & pc-ar \end{smallmatrix} \right)$ equals
$$
X\left(\begin{array}{c|ccc}
pd & (d,a) & -q (d,a) + \beta pb & -r (d,a) + \beta pc \\
(d,a) & 0 & b & c
\end{array}\right).
$$
Note that in general the previous space is not represented by a normalized type. To obtain its normalized one, follow the processes described in (I.1.3) and (I.1.9) of \cite{Martin11PhD}.
\end{ex}

\subsection{Intersection theory on V-manifolds}

The notion of Cartier and Weil $\mathbb{Q}$-divisors coincide on $V$-manifolds and thus a rational intersection theory can be developed for $\mathbb{Q}$-Weil divisors using the theory of line bundles. This intersection multiplicity was first introduced by Mumford for normal surfaces, see~\cite{Mumford61}. Recently in~\cite{AMO11b} explicit formulas for weighted blow-ups and weighted projective planes was calculated.

\begin{prop}\label{formula_self-intersection}
Let $\pi: \widehat{X} \to X$ be the $(p,q)$-blow-up at a point of type $(d;a,b)$ as in~\eqref{w-blow-up_intro}. Consider two $\mathbb{Q}$-divisors $C$ and $D$ on $X(d;a,b)$. Then,
$$
\begin{array}{lcl}
\displaystyle {\rm (1)} \ E \cdot \pi^{*}(C) = 0, & \quad & \displaystyle {\rm (4)} \ E^2 = - \, \frac{e^2}{dpq}, \\
\displaystyle {\rm (2)} \ \pi^{*}(C) = \widehat{C} + \frac{\nu}{e} E, && \displaystyle {\rm (5)} \ \widehat{C} \cdot \widehat{D} = C \cdot D - \frac{\nu \mu}{dpq}, \\
\displaystyle {\rm (3)} \ E \cdot \widehat{C} = \frac{e \nu}{d p q}, && \displaystyle {\rm (6)} \ \widehat{D}^2 = D^2 - \frac{\mu^2}{dpq} \quad \text{($D$ compact)},
\end{array}
$$
where $\nu$ and $\mu$ denote the $(p,q)$-multiplicities of $C$ and $D$ at $P$, i.e.~$x$ (resp.~$y$) has $(p,q)$-multiplicity $p$ (resp.~$q$).
\end{prop}

\begin{prop}\label{bezout_th_P2w-mu_d}
Let us denote by $m_1$, $m_2$, $m_3$ the determinants of the three minors of order $2$ of the matrix $\big( \begin{smallmatrix} p & q & r \\ a & b & c \end{smallmatrix} \big)$. Assume $\gcd(p,q,r) = 1$ and denote $e = \gcd ( d, m_1, m_2, m_3 )$. Consider $\P^2_\w$ the weighted projective plane with $\w=(p,q,r)$. Then, the intersection number of two $\mathbb{Q}$-divisors on the quotient $\P^2_{\w}(d;a,b,c) := \P^2_{\w} / \mu_d$ is $D_1 \cdot D_2 = \frac{e}{dpqr} \deg_{\w}(D_1) \deg_{\w}(D_2)$. Moreover, if $|D_1| \nsubseteq |D_2|$, then $D_1 \cdot D_2 = \sum_{P \in |D_1|\cap |D_2|} (D_1 \cdot D_2)_P$ .
\end{prop}

\begin{remark}\label{computation_local_number_V-surface}
To calculate $(D_1 \cdot D_2)_{[(0,0)]}$ the intersection multiplicity of two $\mathbb{Q}$-divisors on $X(d;a,b)$, $\gcd(d,a,b)=1$, consider $\pr: \C^2 \to X(d;a,b)$ and apply the classical local pull-back formula. Denote by $\widetilde{D}_i$ the pull-back divisor of $D_i$ under the projection. Then, $(D_1 \cdot D_2)_{[(0,0)]} = \frac{1}{d} ( \widetilde{D}_1 \cdot \widetilde{D}_2 )_{(0,0)}$.
\end{remark}

\vspace{0.15cm}

Note that the exceptional divisor of the $(p,q,r)$-weighted blow-up at a point of type $(d;a,b,c)$ is naturally isomorphic to $\P^2_{\w} (d;a,b,c)$. Hence this result will help us describe embedded $\Q$-resolutions for YLS.

\subsection{A'Campo's formula for embedded Q-resolutions}

Let $f: (\C^{n+1},0) \to (\C,0)$ be a non-constant analytic function germ defining an isolated singularity and let $H=\{f=0\}$. Given $\pi: X \to (\C^{n+1},0)$ an embedded ${\bf Q}$-resolution of $(H,0)$, consider $E_1, \ldots, E_s$ the irreducible components of the exceptional divisor and $\widehat{H}$ the strict transform.

One writes $E_0 = \widehat{H}$ and $S=\{0,1,\ldots,s\}$ so that the stratification of $X$ associated with the $\mathbb{Q}$-normal crossing divisor $\pi^{-1}(H) = \bigcup_{i \in S} E_i$ is defined by setting
\begin{equation*}
  E_{I}^\circ := \Big( \cap_{i \in I} E_i \Big) \setminus \Big( \cup_{i\notin I} E_i \Big),
\end{equation*}
for a given possibly empty set $I\subseteq S$.

Let $X = \bigsqcup_{j\in J} Q_j$ be a finite stratification on $X$ given by its quotient singularities such that the local equation of $g := f\circ \pi\,$ at $P \in E_I^{\circ} \cap Q_j$ is of the form
$$
x_0^{m_0} \cdot \ldots \cdot x_k^{m_k}:\, X(\bd; A) := \C^{n+1}/\mu_{\bd} \longrightarrow \C, \quad (0 \leq k \leq n)
$$
and the multiplicities $m_i$'s and the action $\mu_{\bd}$ are the same along each stratum~$E_I^{\circ} \cap Q_j$, i.e.~they do not depend on the chosen point $P \in E^{\circ}_{I} \cap Q_j$.

\begin{defi}\label{def_mult_1}
Using the previous notation the \emph{multiplicity} of $E^{\circ}_{\{i\}} \cap Q_j$ is defined as
\begin{equation*}
m(E^{\circ}_{\{i\}} \cap Q_j) =  \frac{m}{L} \ \in \ \N,
\end{equation*}
where $L = \lcm \left( \frac{d_0}{\gcd(d_0,a_{00})},\ldots,\frac{d_r}{\gcd(d_r,a_{r0})} \right)$ and $x_0^m : X(\bd;A) \to \C$ is the equation of the exceptional divisor at any point $P \in E^{\circ}_{\{i\}} \cap Q_j$.
\end{defi}

Let us denote $\check{E}_{i,j} := E_{\{i\}}^{\circ} \cap Q_j$ and $m_{i,j} := m(\check{E}_{i,j})$. The following result is nothing but the generalization of A'Campo's formula in this setting~\cite{Martin11}.

\begin{theo}\label{ATH2}
The characteristic polynomial of the complex monodromy of $(H,0) \subset (\C^{n+1},0)$ is ($i=1,\ldots,s$,\, $j\in J$)
$$
  \Delta(t) = \left[\frac{1}{t-1} \prod_{i, j} \left(t^{m_{i,j}}-1
  \right)^{\chi( \check{E}_{i,j} )}\right]^{(-1)^n}
$$
and thus the Milnor number is $\displaystyle\mu = (-1)^n
\Big[-1+\sum_{i, j} m_{i,j} \cdot \chi( \check{E}_{i,j}) \Big]$. 
\end{theo}

\section{Preparations for the \textbf{Q}-Resolution of SIS}\label{sec_preparations_SIS}

These singularities have been introduced by Luengo and also appear in a paper by Stevens, where the $\mu$-constant stratum is studied, see \cite{Luengo87} and \cite{Stevens89} respectively. Afterward Artal described in his PhD thesis \cite{Artal94} an embedded resolution of such singularities using blow-ups at points and rational curves.

Here an embedded $\Q$-resolution is given and particularly it is proven that only weighted blow-ups at points are needed. By contrast, the final ambient space obtained has abelian quotient singularities.

Let $(V,0)$ be a SIS in $(\C^3,0)$ defined by a holomorphic function $f:U\to \C$. As above, denote by $m$ the multiplicity of $V$, and ${\bf C}$ the tangent cone. Let $\pi_0: \widehat{U} \to U$ be the blow-up at the origin. Recall that the total transform is the divisor $\pi_0^{*}(V) = \widehat{V} + m E_0$, where $\widehat{V}$ is the strict transform of $V$, and $E_0$ is the exceptional divisor of $\pi_0$. The intersection $\widehat{V}\cap E_0$ is identified with the tangent cone of the singularity, see Figure~\ref{fig_step_zero}.

\begin{figure}[h t]
\centering
\includegraphics{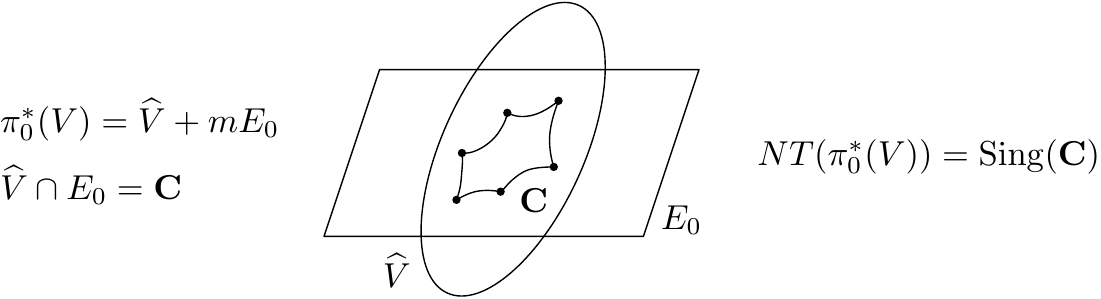}
\caption{Step 0 in the embedded ${\bf Q}$-resolution of $(V,0)$.}
\label{fig_step_zero}
\end{figure}

Let us consider $P\in \widehat{V}\cap E_0 = {\bf C}$. After linear change of coordinates we can assume that $P = ((0,0,0),[0:0:1]) \equiv [0:0:1] \in {\bf C}$. Take a chart of $\widehat{U}$ around $P$ where $z=0$ is the equation of $E_0$ and the blowing-up takes the form
$$
  (x,y,z) \stackrel{\pi_0}{\longmapsto} (x z, y z, z).
$$

Then the equation of $\widehat{V}$ is
$$
  \widehat{V} :\quad f_{m}(x,y,1) + z \Big[ f_{m+1}(x,y,1) + z f_{m+2}(x,y,1) +
  \cdots \Big] = 0.
$$

Two cases arise: if $P$ is smooth in the tangent cone, then $\widehat{V}$ is also smooth at $P$ and the intersection with $E_0$ at that point is transverse; otherwise, i.e $P\in \Sing({\bf C})$, the SIS condition $\Sing({\bf C})\cap V(f_{m+1}) = \emptyset$ implies that the previous expression in brackets is a unit in the local ring $\C \{x,y,z \}$ and, in particular, $\widehat{V}$ is still smooth. Now the order of $f_m(x,y,1)$ is greater than or equal to 2 and the intersection $\widehat{V}\cap E_0$ is not transverse at $P$.

We summarize the previous discussion in the following result, which is actually the step zero in the resolution of \cite{Artal94}.

\begin{lemma}[Step 0]\label{step_zero}
Let $P\in {\bf C}$ be a point in the tangent cone. Then $\widehat{V}$ is smooth in a neighborhood of $P$.

Moreover, the surfaces $\widehat{V}$ and $E_0$ intersect transversely at $P$ if and only if $P$ is a smooth point in ${\bf C}$. Otherwise, i.e.~$P\in \Sing({\bf C})$, there exist local analytic coordinates around $P$ such that the equations of the exceptional divisor and the strict transform are of the form
\begin{eqnarray*}
& E_0 :& z = 0\,;\\[0.1cm]
& \widehat{V} :& z+h(x,y) = 0\,,
\end{eqnarray*}
where $h(x,y)=0$ is an equation of ${\bf C}$ and its order is at least 2.
\end{lemma}

\section{Construction of the Embedded \textbf{Q}-Resolution}\label{complete_description_SIS}

Now we proceed to construct the full ${\Q}$-resolution of $(V,0)$. By the preceding lemma, the set of points where $\pi^{*}_0(V)$ is not a normal crossing divisor is finite, namely $\Sing({\bf C})$. Therefore the next step in the resolution of $(V,0)$ is to blow up those points. Let us fix $P\in \Sing({\bf C})$ and consider local coordinates as in Lemma \ref{step_zero}. Even though many objects that appear in this section depend on $P$, to simplify notation, it is omitted if no confusion seems likely to arise. 

\begin{defi}
Given a divisor $D$, the set of points where $D$ is not a normal crossing divisor is called the {\em locus of non-transversality} of $D$ and it is denote by $NT(D)$.
\end{defi}

In our case, the locus of non-transversality after the blowing-up at the origin of $(V,0)$ is $NT(\pi_0^{*}(V)) = \Sing({\bf C})$.


The following result is the first step in a sequence of blow-ups. We adopt the convention to write the exceptional divisors appearing in the tangent cone in calligraphy letter, while normal letter is used for the divisors in the resolution of $(V,0)$.


Also, the objects coming from the blowing-up at $P_a \neq P$ (resp.~$P$) are indexed by the corresponding subindex~$a$ (resp.~the number $1$). Finally, recall that the strict transform of a divisor is denoted again by the same letter as the own divisor.

\begin{lemma}[Step 1]\label{step1}
Let $(p_1,q_1)\in \N^2$ be two positive coprime numbers. Let $\varpi_1$ be the weighted blow-up at $P\in {\bf C}$ with respect to~$(p_1,q_1)$. Denote by $\E_1$ its exceptional divisor and by $\nu_1$ the $(p_1,q_1)$-multiplicity of ${\bf C}$ at $P$.

Consider $\pi_1$ the $(p_1,q_1,\nu_1)$-weighted blow-up at $P$ in dimension~$3$ and $E_1$ the corresponding exceptional divisor. Then, the total transform of $\pi^{*}_0(V)$ verifies:
\begin{enumerate} 
\item $\pi_1^{*} \pi^{*}_0 (V) = \widehat{V} + m E_0 + (m+1)\nu_1 E_1$,
\item $NT(\pi_1^{*}\pi^{*}_0(V)) = NT(\varpi_1^{*}({\bf C}))$.
\end{enumerate}
\end{lemma}

\begin{proof}
Let us start by blowing up the point $P\in {\bf C}$ with respect to the weight vector $(p_1,q_1)$, $\gcd(p_1,q_1)=1$, in the tangent cone. Consider the local coordinates of Lemma~\ref{step_zero} around~$P$ so that the equation of~${\bf C}$ is~$h(x,y)=0$; thus $\nu_1= \ord_{(p_1,q_1)} h(x,y)$.

The ambient space obtained has two cyclic quotient singular points corresponding to the origin of each chart and located at the exceptional divisor~$\mathcal{E}_1$. The latter can be identified with the usual projective line $\P^1(p_1,q_1) \simeq \P^1$ under the map $[x:y] \mapsto [x^{q_1}:x^{p_1}]$, and it has self-intersection $\frac{-1\ \ }{p_1 q_1}$ by Proposition~\ref{formula_self-intersection}. Using the charts described in Example~\ref{blowup_dim2},
$$
\begin{array}{c|c}
\text{1st chart\ } & X(p_1;-1,q_1) \ \longrightarrow \ \, \widehat{\C}^2(p_1,q_1), \\[0.2cm]
& \,[(x,y)] \ \mapsto \ \big((x^{p_1},x^{q_1} y),[1:y]_{(p_1,q_1)}\big);
\end{array}
$$
$$
\begin{array}{c|c}
\text{2nd chart} & X(q_1;p_1,-1) \ \longrightarrow \ \, \widehat{\C}^2(p_1,q_1), \\[0.2cm]
& \, [(x,y)] \ \mapsto \ \big((x y^{p_1}, y^{q_1}),[x:1]_{(p_1,q_1)}\big);
\end{array}
$$
one obtains the following equations for the divisor $\varpi_1^{*}({\bf C}) = {\bf C} + \nu_1 \E_1$, see Figure~\ref{step1_C}.
$$
X(p_1;-1,q_1) \supseteq \begin{cases} \E_1: & x=0;\\ {\bf C}: & h_1(x,y)=0, \end{cases}
$$
$$
X(q_1;p_1,-1) \supseteq \begin{cases} \E_1: & y=0;\\ {\bf C}: & h_2(x,y)=0. \end{cases}
$$

Note that $h_1(x,y)$ and $h_2(x,y)$ are not functions on the previous quotient spaces but they define a zero set, since they satisfy
\begin{equation}\label{equation_h1h2}
h_1 (\xi_{p_1}^{-1} x, \xi_{p_1}^{q_1} y) = \xi_{p_1}^{\nu_1} h_1 (x,y), \quad h_2 (\xi_{q_1}^{p_1} x, \xi_{q_1}^{-1} y) = \xi_{q_1}^{\nu_1} h_2 (x,y).
\end{equation}

\begin{figure}[h t]
\centering
\includegraphics{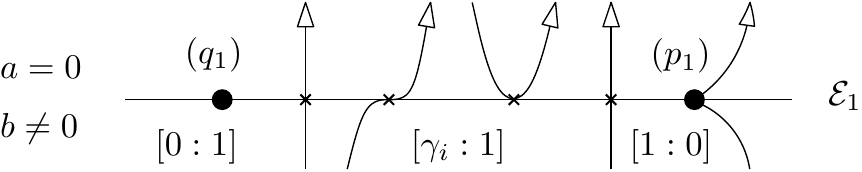}
\caption{Step $1$ in the embedded ${\bf Q}$-resolution of $({\bf C},P)$.}
\label{step1_C}
\end{figure}

Also, if the sum $h = h_{\nu_1} + h_{\nu_1+l} + \cdots$ is the decomposition of $h(x,y)$ into $(p_1,q_1)$-homogeneous parts, then $h_1(0,y) = h_{\nu_1}(1,y)$, $h_2(x,0) = h_{\nu_1}(x,1)$, and the (global) equation of ${\bf C}\cap \E_1 \subset \P^1{(p_1,q_1)}$ is of the form

\begin{equation*}
h_{\nu_1}(x,y) = x^{a} y^{b} \prod_{i} (x^{q_1} - \gamma_i^{q_1} y^{p_1})^{e_i} = 0.
\end{equation*}

Thus the intersection multiplicity of $\E_1$ and ${\bf C}$ at the point $[\gamma_i:1]$ is $e_i$, while it is $\frac{a}{\, q_1}$ (resp.~$\frac{b}{\, p_1}$), not necessarily an integer, at the singular point $[0:1]$ (resp.~$[1:0]$), see Remark~\ref{rem_intNum_step1} below.

Now describe the weighted blow-up at $P$ with respect to $(p_1,q_1,\nu_1)$ in dimension $3$. The new space has in general two (not three because $p_1$ and $q_1$ are coprime) cyclic quotient singular lines, each of them isomorphic to~$\P^1$, and located at the new exceptional divisor $E_1$. They correspond to the lines at infinity $x=0$ and $y=0$ of $E_1 = \P^2(p_1,q_1,\nu_1)$.

As an abstract space, $E_1$ contains two singular points and it is isomorphic to another weighted projective plane as the following expression shows, see Equation~\eqref{propPw},
$$
\begin{array}{c c c}
\P^2(p_1,q_1,\nu_1) &\longrightarrow& \P^2 \Big( \frac{p_1}{(p_1,\nu_1)}, \frac{q_1}{(q_1,\nu_1)}, \frac{\nu_1}{(p_1,\nu_1)(q_1,\nu_1)} \Big), \\[0.25cm]
\, [x:y:z] &\mapsto& [ x^{(q_1,\nu_1)} : y^{(p_1,\nu_1)} : z ].
\end{array}
$$

The multiplicity of $E_1$ is the sum of the $(p_1,q_1,\nu_1)$-multiplicities, in our local coordinates, of the components of the divisor $\pi^{*}_0(V)$ that pass through~$P$, that is $\nu_1 m + \nu_1 = (m+1)\nu_1$. Hence the total transform is the divisor
$$
\pi_1^{*} \pi^{*}_0(V) = \widehat{V} + m E_0 + (m+1)\nu_1 E_1.
$$

The equations in the three charts are given in the table below. Note that the cyclic quotient spaces are represented by normalized types, since $\gcd(p_1,q_1,\nu_1)=1$, see Example~\ref{blowup_dim3_smooth}. \vspace{0.05cm}
$$
\begin{array}{|c|c c c|c c|}
\hline
&& X(p_1;-1,q_1,\nu_1) &&& X(q_1;p_1,-1,\nu_1) \\
(x,y,z) \stackrel{\pi_1}{\longmapsto} && (x^{p_1}, x^{q_1} y, x^{\nu_1} z) &&& (x y^{p_1}, y^{q_1}, y^{\nu_1} z) \\
\hline
E_0 && z=0 &&& z=0 \\
E_1 && x=0 &&& y=0 \\
\widehat{V} && z+h_1(x,y)=0 &&& z+h_2(x,y)=0 \\
\hline
\end{array}
$$

\vspace{0.05cm}

$$
\begin{array}{|c|c c|}
\hline
&& X(\nu_1;p_1,q_1,-1)\\
(x,y,z) \stackrel{\pi_1}{\longmapsto} && (x z^{p_1}, y z^{q_1}, z^{\nu_1})\\
\hline
E_0 && -\\
E_1 && z=0\\
\widehat{V} && 1 + h_{\nu_1}(x,y) + z^l h_{\nu_1+l}(x,y) + \cdots = 0\\
\hline
\end{array}
$$

\vspace{0.25cm}

Using the automorphism on $X(p_1; -1, q_1, \nu_1)$ defined by $[(x,y,z)] \mapsto [( x,y,z+h_1(x,y) )]$, which is well defined due to~\eqref{equation_h1h2}, one sees that both $E_0$ and $\widehat{V}$ intersect transversely $E_1$. The equations of these intersections are given by
\begin{align*}
E_0\cap E_1 & = \{z=0\}, \\
\widehat{V}\cap E_1 & = \{ z+h_{\nu_1}(x,y) = 0\},
\end{align*}
as projective subvarieties in $E_1 = \P^2(p_1,q_1,\nu_1)$.


By Proposition~\ref{bezout_th_P2w-mu_d}, these smooth projective curves are two sections of $E_1$ with self-intersection $\frac{\nu_1}{p_1 q_1}$. They meet at $\#({\bf C}\cap \E_1)$ points with exactly the same intersection number as in ${\bf C}\cap \E_1$, that is, for $P \in {\bf C} \cap \mathcal{E}_1 \equiv \widehat{V} \cap E_0 \cap E_1$, one has
\begin{equation}\label{inter_SIS}
\left( E_0 \cap E_1, \widehat{V} \cap E_1;\, E_1 \right)_P = \left( {\bf C}, \mathcal{E}_1;\, \widehat{\C}^2_{(p_1,q_1)} \right)_P.
\end{equation}

On the other hand, the intersection of the total transform with $E_0$ produces an identical situation to the tangent cone. All these statements follow from the equations above. In Figure~\ref{step1_V}, we see the intersection of the divisor $\pi^{*}_1 \pi^{*}_0 (V)$ with $E_0$ and~$E_1$, respectively.

\begin{figure}[h t]
\centering
\includegraphics{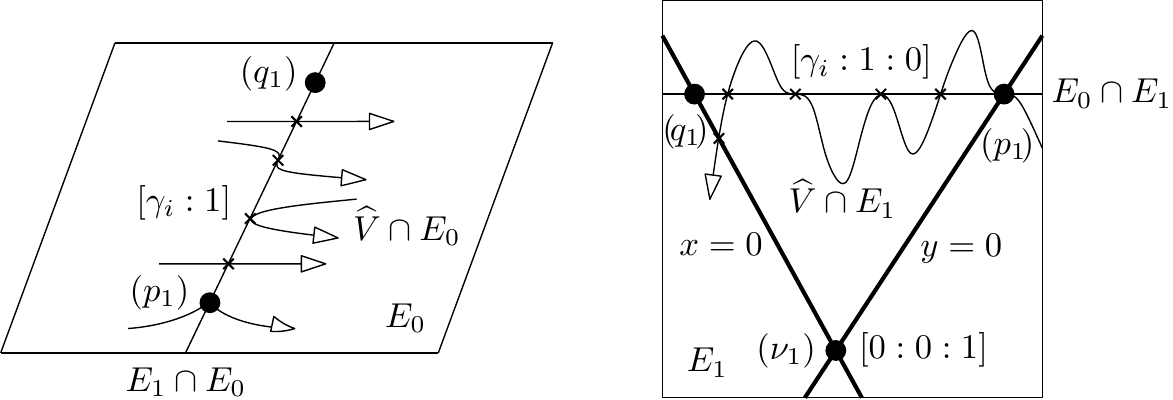}
\caption{Step $1$ in the ${\bf Q}$-resolution of $(V,0)$.}\label{step1_V}
\end{figure}

Finally, the triple points of the total transform in dimension $3$ are identified with the points of ${\bf C}\cap \E_1$ and, by~(\ref{inter_SIS}), the intersection at one of those points is transverse if and only if so is it in dimension $2$. This concludes the proof.
\end{proof}

\begin{remark}\label{rem_intNum_step1}
To study the curves $\{z=0 \}$ and $\{z + h_{\nu_1}(x,y) = 0\}$ in $\P^2(p_1,q_1,\nu_1)$ at the point $[0:1:0]$, one chooses the second chart of the weighted projective plane and obtains the local equations $z=0$ and $z+x^a=0$ around the origin of $X(q_1;p_1,\nu_1)$.
The intersection multiplicity at that point is $a/q_1$, although the quotient space is not represented by a normalized type, see Remark~\ref{computation_local_number_V-surface}. Analogous considerations follow for the points $[\gamma_i:1:0]$ and $[1:0:0]$. This fact was used to prove~(\ref{inter_SIS}).
\end{remark}

\vspace{0.1cm}

\begin{remark}
The curve $\widehat{V} \cap E_1$ meets the line $x=0$ (resp.~$y=0$) in the projective plane $\P^2(p_1,q_1,\nu_1)$ at exactly one point and the intersection is always transverse. If $a=0$ (resp.~$b=0$), then $\gcd(q_1,\nu_1)=q_1$ (resp.~$\gcd(p_1,\nu_1)=p_1$) and that point is different from the origins, see table with the equations. This is important to obtain transversality in the next steps of the resolution of~$(V,0)$.
\end{remark}


After the first blow-up a very similar situation to Lemma \ref{step_zero} is produced, except that there is a new divisor to be considered and the points where the total transform does not have normal crossings could be singular in the ambient space. The main advantage compared with Artal's resolution~\cite{Artal94} is that in the latter $\nu_1$ blow-ups at points and rational curves were needed to achieve a similar situation.

The next result is the second step in the resolution of $(V,0)$ and it corresponds to the second step in the resolution of $({\bf C},P)$. Fix a point $P_a\in NT(\varpi^{*}_{1}({\bf C}))$ and, to cover all cases, assume~$P_a$ is possibly not smooth in the ambient space.

\begin{lemma}[Step 2]\label{step2}
Let $(p_a,q_a)\in \N^2$ be two positive coprime numbers. Let $\varpi_a$ be the weighted blow-up at $P_a$ with respect to $(p_a,q_a)$. Denote by $\E_a$ its exceptional divisor, $\nu_a$ the $(p_a,q_a)$-multiplicity of ${\bf C}$ at~$P_a$, and $m_a$ the multiplicity of $\E_a$.


Consider $\pi_a$ the $(p_a,q_a,\nu_a)$-weighted blow-up at~$P_a$ in dimension~$3$ and let $E_a$ be the corresponding exceptional divisor. Then, the new total transforms satisfy:
\begin{enumerate}
\item $\displaystyle m_a = \frac{\nu_a + p_a \nu_1}{\gcd(p_1,q_a + p_a q_1)}$,
\item $\varpi_a^{*} \varpi_1^{*} ({\bf C}) = {\bf C} + \nu_1 \E_1 + m_a \E_a$,
\item $\pi_a^{*} \pi_1^{*} \pi_0^{*}(V) = \widehat{V} + m E_0 + (m+1) \nu_1 E_1 + (m+1) m_a E_a$,
\item $NT( \pi_a^{*} \pi_1^{*} \pi_0^{*}(V) ) = NT( \varpi_a^{*} \varpi_1^{*} ({\bf C}) )$.
\end{enumerate}
\end{lemma}

\begin{proof}
To fix ideas assume that $P_a = [1:0] \in {\bf C}\cap \E_1$. The other cases follow analogously.
Let us first describe the $(p_a,q_a)$-weighted blow-up at the point $P_a$ in the tangent cone. Consider local coordinates around $P_a$ so that the equation of $\varpi_1({\bf C}) = {\bf C} + \nu_1 \E_1$ is given by the well-defined function
$$
x^{\nu_1} h_1(x,y)\, :\, X(p_1;-1,q_1) \longrightarrow \C,
$$
where $x=0$ is the exceptional divisor $\E_1$ and $h_1(x,y)=0$ is the strict transform of the curve as in the proof of Lemma \ref{step_zero}. Hence the order at $P_a$ is $\nu_a= \ord_{(p_a,q_a)} h_1(x,y)$.

Also, take $\alpha_1$, $\beta_1$ satisfying the B\'{e}zout's identity $\alpha_1 p_1 + \beta_1 q_1 =1$ so that $X(p_1; -1, q_1) = X(p_1; \beta_1,-1)$ and thus $x^{\nu_1} h_1(x,y)$ also defines a function on the latter quotient space.

Denote $d := \gcd(p_1, q_a + p_a q_1)$. Two new cyclic quotient singularities of orders $\frac{p_1 p_a}{d}$ and $\frac{p_1 q_a}{d}$ appear in the ambient space. They correspond to the origin of each chart and thus located at the new exceptional divisor
$$
\mathcal{E}_a = \P^1_{(p_a,q_a)} \Big/ \mu_{p_1} = \P^1_{(p_a,q_a)}(p_1; -1, q_1),
$$
which has self-intersection $\frac{-d^2\ }{p_1 p_a q_a}$, see Proposition~\ref{formula_self-intersection}.

Let $h_1 = h_{\nu_a} + h_{\nu_a+l} + \cdots$ be the decomposition of $h_1(x,y)$ into $(p_a,q_a)$-homogeneous parts. Denote by $g_1(x,y)$ and $g_2(x,y)$ the unique polynomials such that
$$
h_1(x^{p_a}, x^{q_a} y) = x^{\nu_a} g_1(x,y), \quad
h_1(x y^{p_a}, y^{q_a}) = y^{\nu_a} g_2(x,y).
$$
Then, $g_1(x^{\frac{1}{d}},y)|_{x=0} = g_1(0,y) = h_{\nu_a}(1,y)$, and $g_2(x,y^{\frac{1}{d}})|_{y=0} = g_2(x,0) = h_{\nu_a}(x,1)$. Hence the set of points ${\bf C}\cap \E_a$ is given by the (global) equation
$$
\{ h_{\nu_a}(x,y) = 0\} \subset \P^1_{(p_a,q_a)}(p_1;-1,q_1).
$$

Note that $h_{\nu_a}(x,y)$ is not a function on the previous quotient space but it defines a zero set, since
\begin{equation}\label{equation_hnua}
\begin{array}{l}
  h_{\nu_a} (\xi_{p_1}^{-1} x, \xi_{p_1}^{q_1} y) = \xi_{p_1}^{\nu_1} h_{\nu_a} (x,y), \\[0.25cm]
  h_{\nu_a} (\xi_{p_1}^{\beta_1} x, \xi_{p_1}^{-1} y) = \xi_{p_1}^{-\beta_1 \nu_1} h_1(x,y).
\end{array}
\end{equation}

The multiplicity of the new exceptional divisor $\E_a$ is $m_a = \frac{\nu_a + p_a \nu_1}{d}$. The equations of the total transform $\varpi^{*}_a \varpi^{*}_1 ({\bf C})$ in the two charts are given in the table below, see Example~\ref{blowup_dim2}.

\vspace{0.10cm}

\begin{center}
\begin{tabular}{|ll|c|}
\hline 
\multicolumn{2}{|c|}{Equations of $\varpi^{*}_a \varpi^{*}_1 ({\bf C})$} & Chart\\[0.175cm]
\hline
$\E_a:$ & $x=0$ & \multirow{2}{*}{$X\Big(\frac{p_1 p_a}{d};-1,\frac{q_a + p_a q_1}{d}\Big) \ \longrightarrow \ \, \widehat{\C}^2(p_a,q_a)\Big/\mu_{p_1}$}\\[0.125cm]
$\E_1:$ & $-$ &\\
${\bf C}:$ & $g_1(x^{\frac{1}{d}},y) = 0$ & $\big[(x^d,y)\big] \ \mapsto \ \big[ \big((x^{p_a},x^{q_a} y),[1:y]_{(p_a,q_a)}\big) \big]$\\[0.25cm]
\hline 
$\E_a:$ & $y=0$ & \multirow{2}{*}{$X\Big(\frac{p_1 q_a}{d};\frac{p_a + \beta_1 q_a}{d},-1\Big) \ \longrightarrow \ \, \widehat{\C}^2(p_a,q_a)\Big/\mu_{p_1}$}\\[0.125cm]
$\E_1:$ & $x=0$ &\\
${\bf C}:$ & $g_2(x,y^{\frac{1}{d}}) = 0$ & $\big[(x,y^d)\big] \ \mapsto \ \big[ \big((x y^{p_a},y^{q_a}),[x:1]_{(p_a,q_a)}\big) \big]$\\
\hline
\end{tabular}
\end{center}

\vspace{0.10cm}

Now let us see the behavior of the $(p_a,q_a,\nu_a)$-weighted blow-up at the point $P_a$ in dimension~$3$. In our local coordinates around $P_a = [1:0:0] \in (\widehat{V} \cap E_0) \cap E_1$, the equation of the divisor $\,\pi_1^{*} \pi_0^{*} (V) = \widehat{V} + m E_0 + (m+1) \nu_1 E_1\,$ is given by the function
$$
  z^m x^{(m+1) \nu_1} (z+h_1(x,y))\, :\, X(p_1;-1,q_1,\nu_1) \longrightarrow \C.
$$
Note that $X(p_1;-1,q_1,\nu_1) = X(p_1; \beta_1, -1, -\beta_1 \nu_1)$. Now we use the charts described in Example~\ref{blowup_dim3_singular}.

The ambient space has two new lines of singular points corresponding to the lines at infinity $\{x=0\}$ and $\{y=0\}$ of the exceptional divisor
$$
E_a = \P^2_{(p_a,q_a,\nu_a)} \Big/ \mu_{p_1} = \P^2_{(p_a,q_a,\nu_a)}(p_1; -1, q_1, \nu_1).
$$
Recall that $[0:0:1] \in E_a$ is a quotient singular point not necessarily cyclic.

The multiplicity of $E_a$ is the sum of the $(p_a,q_a,\nu_a)$-multiplicities of the components of the divisor $\pi^{*}_1 \pi^{*}_0(V)$ that pass through $P_a$ divided by $d = \gcd(p_1, q_a + p_a q_1)$, that is,
$$
\frac{\nu_a m + p_a (m+1) \nu_1 + \nu_a}{d} = \frac{(m+1)(\nu_a + p_a \nu_1)}{d} = (m+1) m_a.
$$

To study the locus of non-transversality in a neighborhood of $E_a$, the equations of the total transform are calculated in the following table. Note that the third chart is not given in a normalized form but, as we shall see, it is not needed for our purpose.

\begin{center}
\begin{tabular}{lc}
\hline \multicolumn{2}{c}{1st chart} \\
$E_a:$ $x=0$ & \multirow{3}{*}{$X \bigg( \displaystyle\frac{p_1 p_a}{d}; -1, \frac{q_a + p_a q_1}{d}, m_a \bigg)$}\\
$E_1:$ $-$\\
$E_0:$ $z=0$ &\\
$\widehat{V}:$ $z + g_1(x^{\frac{1}{d}},y) = 0$ & $\big[ (x^d,y,z) \big] \ \mapsto \ \big[ \big((x^{p_a},x^{q_a} y, x^{\nu_a} z), [1:y:z] \big) \big]$\\
\\
\hline
\multicolumn{2}{c}{2nd chart} \\
$E_a:$ $y=0$ & \multirow{3}{*}{$X \bigg( \displaystyle\frac{p_1 q_a}{d}; \frac{p_a + \beta_1 q_a}{d}, -1, \frac{\nu_a - \beta_1 \nu_1 q_a}{d} \bigg)$}\\
$E_1:$ $x=0$\\
$E_0:$ $z=0$ &\\
$\widehat{V}:$ $z + g_2(x, y^{\frac{1}{d}}) = 0$ & $\big[ (x,y^d,z) \big] \ \mapsto \ \big[ \big((x y^{p_a},y^{q_a}, y^{\nu_a} z), [x:1:z] \big) \big]$\\
\\
\hline
\multicolumn{2}{c}{3rd chart} \\
$E_a:$ $z=0$ & \multirow{3}{*}{$X \bigg( \begin{array}{c|c c c} \nu_a & p_a & q_a & -1\\ p_1 \nu_a & p_a \nu_1 + \nu_a & q_a \nu_1 - q_1 \nu_a & - \nu_1 \end{array} \bigg)$} \\
$E_1:$ $x=0$\\
$E_0:$ $-$ &\\
$\widehat{V}:$ $1 + \frac{h_1 (x z^{p_a}, y z^{q_a})}{z^{p_a}} = 0$ & $\big[ (x,y,z) \big] \ \mapsto \ \big[ \big((x z^{p_a}, y z^{q_a}, z^{\nu_a}), [x:y:1] \big) \big]$\\
\\ \hline
\end{tabular}
\end{center}

The divisor $m E_0 + (m+1) \nu_1 E_1 + (m+1) m_a E_a$ has clearly normal crossings. Since the polynomial $x^{\nu_1} y^{m_a} g_2 (x, y^{\frac{1}{d}})$ defines a function on the quotient space $X(\frac{p_1 q_a}{d}; \frac{p_a + \beta_1 q_a}{d}, -1)$, the following map is a well-defined automorphism on the corresponding cyclic quotient space
$$
\displaystyle X \Big( \frac{p_1 q_a}{d}; \frac{p_a + \beta_1 q_a}{d}, -1, \frac{\nu_a - \beta_1 \nu_1 q_a}{d} \Big), \ [(x,y,z)] \longmapsto [(x,y,z+g_2(x,y^{\frac{1}{d}})]
$$
and hence the divisor $\widehat{V} + (m+1) \nu_1 E_1 + (m+1) m_a E_a$ has also normal crossings.

Only the intersection $\widehat{V} \cap E_0 \cap E_a$ has to be studied. To do so, we consider the curves $E_0 \cap E_a = \{ z=0 \}$ and $\widehat{V} \cap E_a = \{ z + h_{\nu_a}(x,y) = 0 \}$ as subvarieties in $E_a = \P^2_{(p_a,q_a,\nu_a)}(p_1; -1, q_1, \nu_1)$. The first two charts of the latter space are respectively isomorphic to
$$
  X \Big( \frac{p_1 p_a}{d}; \frac{q_a + p_a q_1}{d}, m_a \Big), \qquad
  X \Big( \frac{p_1 q_a}{d}; \frac{p_a + \beta_1 q_a}{d}, \frac{\nu_a - \beta_1 \nu_1 q_a}{d}
  \Big).
$$

By Proposition~\ref{bezout_th_P2w-mu_d}, these smooth projective curves are two sections of $E_a$ with self-intersection number $\frac{\nu_a d}{p_1 p_a q_a}$; note that
$$
  \gcd \Big(p_1,\, q_a + p_a q_1,\, \nu_a + p_a \nu_1,\,  q_1 \nu_a - \nu_1 q_a \Big) = d,
$$
which is the greatest common divisor needed in the proposition mentioned above.

Now working as in Remark~\ref{rem_intNum_step1}, see also Remark~\ref{computation_local_number_V-surface}, one sees that they meet at $\# ( {\bf C} \cap \E_a )$ points with exactly the same intersection multiplicity as in the latter, that is, for~$P \in {\bf C} \cap \mathcal{E}_a \equiv \widehat{V} \cap E_0 \cap E_a$, one has
\begin{equation}\label{inter_SIS_2}
\left( E_0 \cap E_a, \widehat{V} \cap E_a;\, E_a \right)_P = \left( {\bf C}, \mathcal{E}_a;\, \widehat{\C}^2_{(p_a,q_a)} \big/ \mu_{p_1} \right)_P.
\end{equation}

As in the first step, the intersection of the total transform with $E_0$ produces an identical situation to the tangent cone. Also, note that Figures \ref{step1_C} and \ref{step1_V} can also be used to illustrate the general situation here. The main difference is that the line at infinity $\{ x=0 \} \subset E_a$ coincides with $E_1 \cap E_a$ and thus the point $[0:0:1] \in E_a$ belongs to two divisors.

Now, to finish, observe that the triple points $\widehat{V} \cap E_0 \cap E_a$ of the total transform in dimension $3$ are identified with the points of ${\bf C}\cap \E_a$ and, by~\eqref{inter_SIS_2}, the intersection at one of those points is transverse if and only if so is it in dimension $2$.
\end{proof}

\begin{remark}\label{ma_is_integer}
Note that if $x^k g_1(x,y) : X(e;-1,r) \to \C$ defines a function and $x\nmid g_1(x,y)$, then $d := \gcd(e,r)$ divides $k$ and $g_1(x^{\frac{1}{d}},y)$ is a polynomial. This implies, in particular, that $m_a$ is an integer since the polynomial $x^{\nu_a + p_a \nu_1} g_1(x,y)$ defines a function on $X(p_1 p_a; -1, q_a + p_a q_1)$.
\end{remark}

\vspace{0.1cm}

\begin{remark}\label{remark_acampo_stepa}
If $y\nmid h_{\nu_a} (x,y)$, or equivalently $\E_a \ni [1:0] \notin {\bf C}$, then $p_a | \nu_a$ and $p_1 | (\nu_1 + \frac{\nu_a}{p_a})$; consequently, $\gcd( \frac{p_1 p_a}{d}, m_a) = \frac{p_1 p_a}{d}$.

Indeed, assume that $h_{\nu_a}(x,y) = x^{e_0} y^{e_\infty} \prod_{i\geq 1} (x^{q_a}- \gamma_{i} y^{p_a})^{e_i}$. Then, its order is $\nu_{a} = e_0 p_a + e_\infty q_a + p_a q_a \sum_i e_i$. By~\eqref{equation_hnua}, the following two expressions are equal:
\begin{align*}
h_{\nu_a} ( \xi_{p_1}^{-1} x, \xi_{p_1}^{q_1} y ) & = \xi_{p_1}^{-e_0+e_{\infty} q_1} x^{e_0} y^{e_\infty} \prod ( \xi_{p_1}^{-q_a} x^{q_a} - \xi_{p_1}^{q_1 p_a} \gamma_i y^{p_a} )^{e_i} = \\
& = \xi_{p_1}^{-e_0 + e_\infty q_1 - q_a \sum_i e_i} x^{e_0} y^{e_\infty} \prod (x^{q_a} -  \xi_{p_1}^{q_1 p_a + q_a} \gamma_{i} y^{p_a})^{e_i}, \\[0.1cm]
\xi_{p_1}^{\nu_1} h_{\nu_a} (x,y) & = \xi_{p_1}^{\nu_1} x^{e_0} y^{e_\infty} \prod (x^{q_a} - \gamma_{i} y^{p_a})^{e_i}.
\end{align*}
Hence $p_1$ divides $\nu_1 + e_0 - e_\infty q_1 + q_a \sum_i e_i$. In the case $e_\infty = 0$, the latter number is $\nu_1 + \frac{\nu_a}{p_a}$ and the claim follows.

Anologously, if $x\nmid h_{\nu_a}(x,y)$ ($\Leftrightarrow \E_a \ni [0:1] \notin {\bf C} \Leftrightarrow e_0 = 0$), then one has that $q_a | \nu_a$ and $p_1 | (\frac{\nu_a}{q_a} - \beta_1 \nu_1)$; consequently, $\gcd(\frac{p_1 q_a}{d}, \frac{p_a + \beta_1 q_a}{d}) = \frac{p_1 q_a}{d}$.
\end{remark}

\vspace{0.1cm}

\begin{remark}
Although the third chart, say $X_3$, is not in general a cyclic quotient space, there are a couple of situations where it is.
\begin{itemize}
\item If $\gcd(\nu_1, \nu_a) = 1$, then the action given by the second row includes the first one and thus $X_3$ is just $\C^3$ under the second row action.
\item Also if $\gcd (p_1, \nu_1) = 1$ and $\lambda$ is the inverse of $\nu_1$ modulo $p_1$, then $X(p_1; -1, q_1, \nu_1)$ can be written in the form $X(p_1; \lambda, - \lambda q_1, -1)$ and
thus $X_3 = X(p_1 \nu_a; p_a + \lambda \nu_a, q_a - \lambda q_1 \nu_a, -1)$.
\end{itemize}
\end{remark}

Let $\Gamma$ and $\Gamma_{+}$ be the dual graphs associated with the total transform and the exceptional divisor, after having computed an embedded $\Q$-resolution of $({\bf C}, P)$, respectively. Denote by $S(\Gamma)$ and $S(\Gamma_{+})$ the sets of their vertices. The classical partial order on $S(\Gamma_{+})$ is denoted by $\preccurlyeq$.

The locus of non-transversality after the last blow-up in dimension~3 is identified with the locus of non-transversality in the resolution of~$({\bf C},P)$. Each of these points corresponds to a weighted blow-up in the resolution of the tangent cone, that is, to a vertex of $\Gamma_{+}$. Thus in the next step we need to blow-up those points to produce a similar situation. Again the same operation will be applied to the points where the total transform is not a normal crossing divisor. These points will also be associated with vertices of $\Gamma_{+}$.

The following result is proven by induction on $S(\Gamma_{+})$ using the relation~$\preccurlyeq$. Lemma~\ref{step1} is the first step in the induction. The proof of Lemma~\ref{step2} tells us the way to show the general case. Let $b\in S(\Gamma_{+})$ be a vertex such that $P_b$ belongs to the locus of non-transversality of the total transform. As usual, denote by $\E_b$ the exceptional divisor appearing after blowing up the point $P_b$.

\begin{prop}[Step $b$]\label{step_b_SIS}
Let $\varpi_b$ be the $(p_b,q_b)$-weighted blow-up at $P_b$ with $b\in S(\Gamma_{+})$. Denote by $\E_b$ its exceptional divisor, $\nu_b$ the $(p_b,q_b)$-mul\-ti\-plic\-i\-ty of ${\bf C} \subset \C^2$, and $m_b$ the multiplicity of $\E_b$.

Consider $\pi_b$ the $(p_b,q_b,\nu_b)$-weighted blow-up at $P_b$ in dimension $3$ and $E_b$ the corresponding exceptional divisor. Then, after blowing up the point~$P_b$, the new total transform verifies:
\begin{enumerate}
\item The exceptional divisor $E_b$ is isomorphic to $\P^2 (p_b, q_b, \nu_b) / \mu_e$ and its multiplicity equals $(m+1) m_b$. In general, the lines at infinity $\{x=0\}$ and $\{y=0\}$ are quotient singular in the ambient space and the point $[0:0:1]$ is the only one which may be non-cyclic. By contrast, the stratum $\{ z=0 \} \setminus \{ [0:1:0], [1:0:0] \} \subset E_b $ does not contain singular points of the ambient space.

\item Let $a$ be a vertex such that $a\prec b$. Then, $E_a \cap E_b \neq \emptyset$ if and only if
$P_b \in \E_a$. In such a case, $E_a \cap E_b$ is one of the two lines at infinity of $E_b$ different from~$\{z=0\}$. If $P_b \in \E_a \cap \E_{a'}$, $a\neq a'$, then the corresponding lines are different and hence they meet at the point $[0:0:1]$.

\item The intersection of the rest of components with $E_0$ produces an identical situation to the resolution of $({\bf C},P)$, after blowing up the point $P_b$. More precisely,
\begin{align*}
\widehat{V} \cap E_0 & = {\bf C}, \\
E_b \cap E_0 & = \E_b, \\
E_a \cap E_0 & = \E_a, \quad \forall a \preccurlyeq b.
\end{align*}

\item The curves $E_0 \cap E_b = \{z=0\}$ and $\widehat{V} \cap E_b = \{ z + H_{\nu_b}(x,y) = 0 \}$ are two $\big( \frac{- \mathcal{E}_b^2 \nu_b}{d} \big)$-sections of $E_b$ and the intersecting points can be identified with ${\bf C}\cap \E_b$. Moreover, the intersection multiplicity of these two sections at one of those points is the same as in the latter, that is, for $P \in {\bf C} \cap \mathcal{E}_b \equiv \widehat{V} \cap E_0 \cap E_b$, one has
$$
\hspace{1.5cm} \left( E_0 \cap E_b, \widehat{V} \cap E_b;\, E_b \right)_P = \left( {\bf C}, \mathcal{E}_b;\, \widehat{\C}^2_{(p_b, q_b)} \big/ \mu_{e} \right)_P.
$$
If $P_b\in \E_a$, then $E_a \cap E_b$ and $\widehat{V}\cap E_b$ always meet at exactly one point. This point passes through $E_0\cap E_b$ if and only if ${\bf C}\cap \E_a \cap \E_b \neq \emptyset$. This is the case when there exist quadruple points.

\item The locus of non-transversality of the total transform in dimension 3 is identified with the one in the resolution of $({\bf C},P)$. These points belong to $\widehat{V} \cap E_0 \cap E_b = {\bf C} \cap \E_b$ and they correspond to the ones where the curves $E_0 \cap E_b$ and $\widehat{V} \cap E_b$, or equivalently $\E_b$ and ${\bf C}$, do not meet transversely.

\item The strict transform $\widehat{V}$ never passes through $[0:0:1]\in E_b$. In particular, $\widehat{V}$ only contains cyclic quotient singularities.
\end{enumerate}
\end{prop}

\begin{proof}
By induction on $S(\Gamma_{+})$ with respect to $\preccurlyeq$. Lemma~\ref{step1} is base case. As for the inductive step, one proceeds as in the proof of Lemma~\ref{step2}. Assume, by induction, that the local equation of the total transform in the resolution of the tangent cone around~$P_b$ is given by ($\gcd(e,r) = \gcd(e,s) = 1$)
\begin{equation}\label{eq_stepb_C}
x^{m_{a}} y^{m_{a'}} H(x,y)\, :\, X(e; r,s) \longrightarrow \C,
\end{equation}
where ${\bf C} = \{ H(x,y) = 0 \}$ is the equation of the strict transform and the others correspond to the divisors $\E_a$ and $\E_{a'}$ (they may not appear if $m_a$ or $m_{a'}$ equals zero).

Also, the equation of the total transform around $P_b$ in dimension~$3$ is given by the function
\begin{equation}\label{eq_stepb_V}
x^{(m+1) m_{a}}\cdot y^{(m+1) m_{a'}}\cdot z^{m} \big[z + H(x,y) \big]\, :\, X(e; r,s,t) \longrightarrow \C,
\end{equation}
where $\widehat{V} = \{ z + H(x,y) = 0\}$ is the strict transform, $E_0 = \{z=0\}$, and the others are the divisors $E_a$ and $E_{a'}$ (if they exist). Using that both \eqref{eq_stepb_C} and \eqref{eq_stepb_V} are well-defined functions, one has
$$
t + m_{a} \cdot r + m_{a'} \cdot s \in (e).
$$

The verification of the statement is very simple once the local equations of the divisors appearing in the total transform are calculated. The main ideas behind are contained in the proof of Lemmas~\ref{step1} and~\ref{step2}. The details are omitted to avoid repeating the same arguments; only the local equations are given, see table below.

To do so, consider the following data and use the charts described in Examples~\ref{blowup_dim2} and~\ref{blowup_dim3_singular}. As auxiliary results, Propositions~\ref{formula_self-intersection} and~\ref{bezout_th_P2w-mu_d} and Remark~\ref{computation_local_number_V-surface} are also needed.
\begin{alignat*}{3}
\nu_b & = \ord_{(p_b,q_b)} H (x,y) & \hspace{1cm} m_b & = \frac{p_b\cdot m_a + q_b\cdot m_{a'} + \nu_b}{d} \\
d & = \gcd(e,\, p_b\cdot s - q_b\cdot r) &&
\end{alignat*} \vspace{-0.75cm}
\begin{alignat*}{3}
s' r + s & \equiv 0 \quad \text{mod $(e)$} & \hspace{1cm} r' s + r & \equiv 0 \quad \text{mod $(e)$} \\
H_1(x,y) & = \frac{H(x^{p_b}, x^{q_b} y)}{x^{\nu_b}} & \hspace{1cm} H_2(x,y) & = \frac{H(x y^{p_b},y^{q_b})}{y^{\nu_b}}
\end{alignat*}

These are the equations in the resolution of the tangent cone ${\bf C}$ presented as zero sets in the corresponding (abelian) quotient space, cf.~proof of Lemma~\ref{step2}.

\vspace{0.20cm}

\begin{center}
\begin{tabular}{|ll|c|}
\hline 
\multicolumn{2}{|c|}{Equations} & Chart\\[0.175cm]
\hline
$\E_b:$ & $x=0$ & \multirow{3}{*}{$\displaystyle X \bigg(\frac{e p_b}{d};-1,\frac{q_b + s' p_b}{d}\bigg) \ \longrightarrow \ \, \widehat{\C}^2(p_b,q_b)\Big/\mu_{e}$}\\
$\E_{a}:$ & $-$ &\\
$\E_{a'}$ & $y=0$ & \\
${\bf C}:$ & $H_1(x^{\frac{1}{d}},y) = 0$ & $\big[(x^d,y)\big] \ \mapsto \ \big[ \big((x^{p_b},x^{q_b} y),[1:y]_{(p_b,q_b)}\big) \big]$\\[0.25cm]
\hline 
$\E_b:$ & $y=0$ & \multirow{3}{*}{$\displaystyle X \bigg(\frac{e q_b}{d};\frac{p_b + r' q_b}{d},-1 \bigg) \ \longrightarrow \ \, \widehat{\C}^2(p_b,q_b)\Big/\mu_{e}$}\\
$\E_a:$ & $x=0$ &\\
$\E_{a'}:$ & $-$ & \\
${\bf C}:$ & $H_2(x,y^{\frac{1}{d}}) = 0$ & $\big[(x,y^d)\big] \ \mapsto \ \big[ \big((x y^{p_b},y^{q_b}),[x:1]_{(p_b,q_b)}\big) \big]$\\[0.25cm]
\hline
\end{tabular}
\end{center}

\vspace{0.20cm}

In dimension~$3$, the local equations of the total transform are presented as well-defined functions over the corresponding quotient spaces. The notation is self-explanatory to recognize the equation of each divisor.
$$
\begin{array}{r|cl}
\text{1st chart\ } && X \bigg( \displaystyle\frac{e p_b}{d}; -1, \frac{q_b + s' p_b}{d}, \frac{\nu_b + t' p_b}{d} \bigg) \longrightarrow \C \\[0.5cm]
&& x^{(m+1) m_b}\cdot y^{(m+1) m_{a'}}\cdot z^m \big[ z + H_1 (x^{\frac{1}{d}}, y) \big] 
\\ \multicolumn{2}{c}{} \\
\text{2nd chart\ } && X \bigg( \displaystyle\frac{e q_b}{d}; \frac{p_b + r' q_b}{d}, -1, \frac{\nu_b + t'' q_b}{d} \bigg) \longrightarrow \C \\[0.5cm]
&& x^{(m+1) m_a}\cdot y^{(m+1) m_{b}}\cdot z^m \big[ z + H_2 (x, y^{\frac{1}{d}}) \big]
\\ \multicolumn{2}{c}{} \\
\text{3rd chart\ } && X \bigg( \begin{array}{c | c c c} \nu_b & p_b & q_b & -1 \\ e \nu_b & r \nu_b - t p_b & s \nu_b - t  q_b & t \end{array} \bigg) \longrightarrow \C \\[0.5cm]
&& x^{(m+1) m_a} \cdot y^{(m+1) m_{a'}} \cdot z^{(m+1) m_b \cdot d} \Big[ 1 + \frac{H(x z^{p_b}, y z^{q_b})}{z^{\nu_b}} \Big]
\end{array}
$$

Here $t'$ and $t''$ are taken so that $\, t' r + t \equiv 0\,$ and $\, t'' s + t \equiv 0\,$ modulo~$(e)$. The exceptional divisor $E_b$ is identified with $\P^2 (p_b, q_b, \nu_b) / \mu_{e}$ where the action is of type $(e;r,s,t)$, i.e.~$E_b = \P^2_{(p_b,q_b,\nu_b)} (e; r,s,t)$.
\end{proof}

\begin{remark}
Note that the equations after the blowing-up at~$P_b$ around the points where the total transform is not a normal crossing divisor are of the same form as in~(\ref{eq_stepb_C}) and~(\ref{eq_stepb_V}). Hence, by induction, this fact holds for every stage of the resolution.
\end{remark}

\vspace{0.1cm}

\begin{remark}\label{remark_acampo_key}
Let us write $H_{\nu_b}(x,y) = x^{e_0} y^{e_\infty} \prod_{i\geq 1} ( x^{q_b} - \gamma_i y^{p_b} )^{e_i}$. As in Remark \ref{remark_acampo_stepa}, if
$$
y \nmid H_{\nu_b}(x,y) \quad \big( \Longleftrightarrow \E_b \ni [1:0] \notin {\bf C} \Longleftrightarrow e_\infty = 0 \ \big),
$$
then $p_b | \nu_b$ and $e | (\frac{\nu_b}{p_b} + t')$;  consequently, $\gcd(\frac{e p_b}{d}, \frac{\nu_b + t' p_b}{d}) = \frac{e p_b}{d}$. Analogously, $e_0 = 0$ implies $\gcd(\frac{e q_b}{d}, \frac{v_b + t'' q_b}{d}) = \frac{e q_b}{d}$.
\end{remark}

\begin{theo}\label{summery_theorem}
Given an embedded $\Q$-resolution of $({\bf C},P)$ for all $P\in \Sing({\bf C})$, one can construct an embedded $\Q$-resolution of $(V,0)$, consisting of weighted blow-ups at points. Each of these blow-ups corresponds to a weighted blow-up in the resolution of $({\bf C},P)$ for some $P\in \Sing ({\bf C})$, that is, it corresponds to a vertex of $\Gamma_{+}^P$. \hfill $\Box$
\end{theo}

We shall see later that an exceptional divisor $E_a^P$ in the resolution of~$(V,0)$ obtained contributes to the monodromy if and only if so does the corresponding divisor $\E_a^P$ in the tangent cone, see Lemma~\ref{euler_char_SIS} and Theorem~\ref{char_poly_SIS}.
In particular, the weights can be chosen so that every exceptional divisor in the embedded $\Q$-resolution of $(V,0)$, except perhaps the first one $E_0$, contributes to its monodromy.

\section{The Characteristic Polynomial of SIS}\label{sec_char_poly_SIS}

Here we plan to apply Theorem~\ref{ATH2} to compute the characteristic polynomial of the monodromy and the Milnor number of $(V,0)$ in terms of its tangent cone $({\bf C},P)$. Some notation need to be introduced, concerning the stratification of each irreducible component of the exceptional divisor in terms of its quotient singularities.

\vspace{0.25cm}

Given a point $P \in \Sing({\bf C})$, denote by $\varrho^{P}: Y^p \to ({\bf C},P)$ an embedded $\Q$-resolution of the tangent cone. Assume that the total transform is given by
$$
(\varrho^{P})^{*}({\bf C},P) = {\bf C} + \sum_{a \in S(\Gamma_{+}^{P})} m_a^{P} \mathcal{E}_a^{P},
$$
where $\E_a^{P}$ is the exceptional divisor of the $(p_a^P,q_a^P)$-blow-up at a point $P_a$ belonging to the locus of non-transversality. Denote by $\nu_a^P$ the $(p_a^P,q_a^P)$-multiplicity of ${\bf C}$ at $P_a$.

Recall that $\E_a^P$ is naturally isomorphic to $\P^1_{(p_a^P,q_a^P)} / \mu_e$. Using this identification, see Figure~\ref{strata_EaP}, define
$$
\E^P_{a,1} = \E^{P}_a \setminus \{ [0:1], [1:0] \}, \qquad
\E^P_{a,x} = \{ [0:1] \}, \qquad
\E^P_{a,y} = \{ [1:0] \}.
$$
The strata $\check{\E}_{a,j}^P := \E^{P}_{a,j} \setminus \big( \E_{a,j}^{P} \cap \big( \bigcup_{b \neq a} \E_b^P \cup {\bf C} \big) \big)$ for $j=1,x,y$ (see notation just above Theorem~\ref{ATH2}) will be considered in Lemma~\ref{euler_char_SIS}.

\vspace{0.25cm}

Let us see the situation in the superisolated singularity $(V,0)$. Denote by $\rho: X \to (V,0)$ the embedded $\Q$-resolution obtained following Proposition~\ref{step_b_SIS}. Then, the total transform is
$$
  \rho^{*}(V,0) = \widehat{V} + m E_0 + \sum_{\begin{subarray}{c} P \in \Sing({\bf C}) \\ a \in S(\Gamma_{+}^{P})\end{subarray}} (m+1) m_a^{P} E_{a}^{P},
$$
and $E_a^P$ appears after the $(p_a^P,q_a^P,\nu_a^P)$-blow-up at the point $P_a$ (recall that the locus of non-transversality in dimension 2 and 3 are identified).

The divisor $E_a^P$ is naturally isomorphic to $\P^2_{(p_a^P, q_a^P, \nu_a^P)} / \mu_e$. Using this identification, see Figure~\ref{strata_EaP}, define
$$
\begin{array}{l}
E_{a,1}^P = E_a^P \setminus \{ xy = 0 \}, \qquad E_{a,x}^P = \{ x=0 \} \setminus \{ [0:1:0], [0:0:1] \}, \\[0.2cm]
E_{a,y}^P = \{ y=0 \} \setminus \{ [1:0:0], [0:0:1] \}, \hfill E_{a,xy}^P = \{ [0:0:1 \}.
\end{array}
$$

Analogously, one considers $E_{a,xz}^P$ and $E_{a,yz}^P$ so that $E_a^P = \bigsqcup_{j} E_{a,j}^{P}$ really defines a stratification of $E_a^P$. However, these two strata belong to more than one irreducible divisor in the total transform and hence they do not contribute to the characteristic polynomial.
As for $E_0$, according to its quotient singularities, no stratification need to be considered (it is smooth).

The Euler characteristic of $\check{E}_0$ and $\check{E}_{a,j}^P := E^{P}_{a,j} \setminus \big( E_{a,j}^{P} \cap \big( \bigcup_{b \neq a} E_b^P \cup \widehat{V} \big) \big)$ for $j=1,x,y,xy$ (see notation just above Theorem~\ref{ATH2}) as well as its multiplicity are calculated in Lemma~\ref{euler_char_SIS}.

\begin{figure}[h t]
\centering
\includegraphics{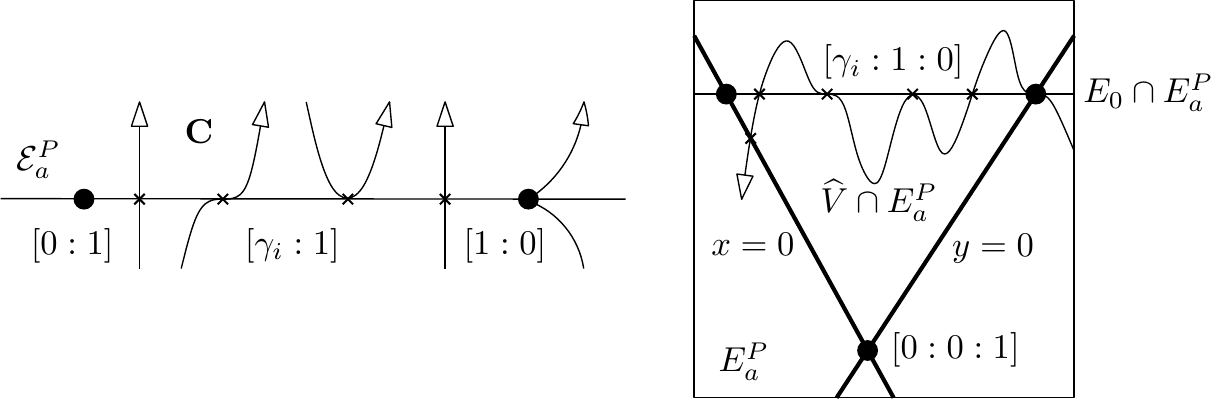}
\caption{Stratification of $\E_a^{P}$ and $E_a^{P}$.}
\label{strata_EaP}
\end{figure}

\begin{lemma}\label{euler_char_SIS}
Using the previous notation, the Euler characteristic and the multiplicity of $\check{E}_0$ are
$$
\chi (\check{E}_0) = \chi (\P^2 \setminus {\bf C}), \qquad m(\check{E}_0) = m.
$$

For the rest of strata of $\check{E}_{a}^{P}$, let us fix a point $P \in \Sing({\bf C})$. Then, one has that
$$
\chi ( \check{E}_{a,j}^{P} ) = \begin{cases} 1 & a = 1,\ j=xy \\ 0 & a \neq 1,\ j = xy \\ - \chi(\check{\E}_{a,j}^{P}) & \forall a,\ j=1,x,y\,; \end{cases}
$$
$$
\chi ( \check{\E}_{a,j}^{P} ) \neq 0 \ \Longrightarrow \ m( \check{E}_{a,j}^{P} ) = \begin{cases} m+1 & a = 1,\ j=xy \\
m(\check{\E}_{a,j}^{P}) \cdot (m+1) & \forall a,\ j=1,x,y. \end{cases}
$$
In fact, $\forall a \in S(\Gamma_{+}^{P})$, $a \neq 1$, the stratum $\check{E}^{P}_{a,xy}$ is empty and, in particular, its Euler characteristic is zero.
\end{lemma}

\begin{proof}
Let $E$ be an irreducible component of the exceptional divisor of $\rho$. Let us travel back in the history of the resolution until the time when $E$ first appears. Consider the space defined by $E$ minus the intersections with the other components at that moment.

Since all the weighted blow-ups have center in that intersections, this space is naturally isomorphic to~$\check{E}$. Using these arguments, we will perform the calculations of the Euler characteristics at the moment when the component appears in the history of the resolution.

For $E_0$, the space $\check{E}_0$ is isomorphic to $E_0 \setminus ( \widehat{V} \cap E_0 )$ which is identified with~$\P^2 \setminus {\bf C}$; its multiplicity is $m$, see Figure~\ref{fig_step_zero} and discussion before Lemma~\ref{step_zero}.


For the rest of the proof the cases $j = 1, x, y, xy$ are treated separately. Let us fix a point $P \in \Sing({\bf C})$ and omit the index ``$P$'' to simplify the notation.

Recall that $E_a = \P^2_{(p_a,q_a,\nu_a)} / \mu_e$, see Proposition~\ref{step_b_SIS}(1). Also Figure~\ref{strata_EaP} will be useful.

\begin{itemize}
\item \underline{$j = xy$}\,:
\end{itemize}

The stratum $E_{a,xy}$ is the point $[0:0:1] \in E_a$. By Proposition~\ref{step_b_SIS}, it belongs to just one divisor if and only if $a \neq 1$, see Lemma~\ref{step1} and its proof. This implies that $\chi( \check{E}_{1,xy} )=1$ and that
$$
  \chi ( \check{E}_{a,xy} ) = 0,
  \qquad \forall a \in S(\Gamma_{+}) \setminus \{ 1 \}.
$$
Following Definition~\ref{def_mult_1}, the multiplicity of $\check{E}_{1,xy}$ is $\frac{(m+1) \nu_1}{\nu_1}$, since the origin $[0:0:1] \in E_1$ is a cyclic quotient singular point of order $\nu_1$, see Lemma~\ref{step1}.

\begin{itemize}
\item \underline{$j = x$}\,:
\end{itemize}

The stratum $E_{a,x}$ is the line $\{ x = 0 \} \subset E_a$. If there is another component of the divisor that passes through $\E_{a,x} = [0:1] \in \E_a$, then one has $\check{\E}_{a,x} = \emptyset$, and either $\check{E}_{a,x} = E_{a,x} \setminus \{  \text{$2$ points} \}$ or $\check{E}_{a,x} = \emptyset$. Otherwise, $\check{\E}_{a,x} = [0:1]$ and $\check{E}_{a,x} = E_{a,x} \setminus \{ \text{$3$ points} \}$, see second part of Proposition~\ref{step_b_SIS}(4). In the case when the Euler characteristic is different from zero, by Remark~\ref{remark_acampo_key}, the multiplicity is
$$
  m( \check{E}_{b,x} ) = \frac{(m+1) m_b}{\gcd(\frac{e q_b}{d}, \frac{\nu_b + t'' q_b}{d} )} = \frac{(m+1) m_b}{\frac{e q_b}{d}} = (m+1) m( \check{\E}_{b,x} ).
$$

The case $j=y$ is exactly the same as $j=x$.

\begin{itemize}
\item \underline{$j = 1$}\,:
\end{itemize}

Consider the projection of $E_{a} \setminus E_{a,xy}$ onto the line $\{ z = 0 \} \equiv \E_{a}$. This map is identified with the morphism
\begin{eqnarray*}
\tau\,: \, \P^2_{(p_a,q_a,\nu_a)}(e; r,s,t) \setminus \{ [0:0:1] \} & \longrightarrow & \P^1_{(p_a,q_a)}(e;r,s), \\
\,[x:y:z] & \mapsto & [x:y].
\end{eqnarray*}
Note that the restriction $\tau|: \check{E}_{a,1} \to \check{\E}_{a,1}$ is a fibration with fiber isomorphic to $\C \setminus \{ \text{$2$ points} \}$ and hence $\chi (\text{fiber}) = -1$.

The multiplicity of the smooth part is  $(m+1) m_a$ in the superisolated singularity while it is $m_a$ in the tangent cone.

To finish observe that in any case, one has that $\chi ( \check{E}_{a,j} ) = - \chi ( \check{\E}_{a,j} )$ and, if they are different from zero, $m( \check{E}_{a,j} ) = (m+1) m( \check{\E}_{a,j} )$. Now the proof is complete.
\end{proof}

\begin{remark}
The Euler characteristic of the complement of a projective plane curve in $\P^2$ is known to be
$$
\chi(\P^2 \setminus {\bf C}) = (m^2 - 3m + 3) - \sum_{P \in \Sing({\bf C})} \mu_P,
$$
see~\cite{Esnault82}, or~\cite{Artal94b} for an elementary proof based on the additivity of the Euler characteristic.
\end{remark}

\begin{theo}\label{char_poly_SIS}
The characteristic polynomial of the complex monodromy of $(V,0)$ is
$$
  \Delta_{(V,0)}(t) = \frac{(t^m-1)^{\chi(\P^2 \setminus {\bf C})}}{t-1}
  \prod_{P\in \Sing({\bf C})} \Delta_{({\bf C},P)} (t^{m+1}),
$$
where $\Delta_{({\bf C},P)}(t)$ denotes the characteristic polynomial of the local complex monodromy of $({\bf C},P)$. \end{theo}

\begin{proof}
Given a point $P \in \Sing({\bf C})$, let us compute the characteristic polynomial of $({\bf C},P)$. Its total transform is
$$
(\varrho^{P})^{*}({\bf C},P) = \widehat{{\bf C}} + \sum_{a \in S(\Gamma_{+}^{P})} m_a^{P} \mathcal{E}_a^{P},
$$
and the stratification associated with each exceptional divisor needed for applying A'Campo's formula is $\check{\mathcal{E}}_a = \check{\mathcal{E}}_{a,1} \sqcup \check{\mathcal{E}}_{a,x} \sqcup \check{\mathcal{E}}_{a,y}$. Then, by Theorem~\ref{ATH2},
\begin{equation}\label{formula_AC_tangent_cone}
\Delta_{({\bf C},P)} (t) = ( t - 1) \prod_{\begin{subarray}{c} a \in S(\Gamma_{+}^{P}) \\ j = 1,\, x,\, y \end{subarray}} ( t^{m(\check{\mathcal{E}}_{a,j}^P)} - 1 )^{-\chi(\check{\mathcal{E}}_{a,j}^P)}.
\end{equation}

Let us see the situation in the superisolated singularity $(V,0)$. The total transform is
$$
  \rho^{*}(V,0) = \widehat{V} + m E_0 + \sum_{\begin{subarray}{c} P \in \Sing({\bf C}) \\ a \in S(\Gamma_{+}^{P})\end{subarray}} (m+1) m_a^{P} E_{a}^{P},
$$
and the corresponding stratification is $\check{E}_{a}^{P} = \check{E}_{a,1}^{P} \sqcup \check{E}_{a,x}^{P} \sqcup \check{E}_{a,y}^{P} \sqcup \check{E}_{a,xy}^{P}$.

By Theorem~\ref{ATH2}, the characteristic polynomial of $(V,0)$ is
\begin{equation}\label{formula_AC_SIS}
\Delta_{(V,0)}(t) = \frac{1}{t-1} ( t^{m(\check{E}_0 )}-1 )^{\chi(\check{E}_0)} \prod_{\begin{subarray}{c} P \in \Sing({\bf C}) \\ a \in S(\Gamma_{+}^{P}) \\ j = 1,\, x,\, y,\, xy \end{subarray}} ( t^{m( \check{E}_{a,j}^P)} - 1 )^{\chi(\check{E}_{a,j}^P)}.
\end{equation}

From Lemma~\ref{euler_char_SIS}, $m(\check{E}_0) = m$ and $\chi(\check{E}_0) = \chi(\P^2 \setminus {\bf C})$, and the latter can be computed combinatorially as indicated in the statement. Let us calculate the contribution of the preceding product for a given point $P \in \Sing({\bf C})$.

Again using Lemma~\ref{euler_char_SIS} and, in particular, the fact that $a \neq 1$ implies $\chi(\check{E}_{a,xy}^{P}) = 0$, one has that
\begin{align*}
& \prod_{\begin{subarray}{c} a \in S(\Gamma_{+}^{P}) \\ j = 1,\, x,\, y,\, xy \end{subarray}} \left( t^{m(\check{E}_{a,j}^P)} - 1 \right)^{\chi(\check{E}_{a,j}^P)} = \\
& \qquad = \underbrace{\left( t^{m(\check{E}_{1,xy}^P)} - 1 \right)^{\chi(\check{E}_{1,xy}^P)}}_{a=1,\ \ j=xy} \prod_{\begin{subarray}{c} a \in S(\Gamma_{+}^{P}) \\ j = 1,\, x,\, y \end{subarray}} \left( t^{m(\check{E}_{a,j}^P)} - 1 \right)^{\chi(\check{E}_{a,j}^P)} \\
& \qquad = \left( t^{m+1} - 1 \right)^1 \prod_{\begin{subarray}{c} a \in S(\Gamma_{+}^{P}) \\ j = 1,\, x,\, y \end{subarray}} \left( t^{(m+1) m(\check{\mathcal{E}}_{a,j}^P)} - 1 \right)^{-\chi(\check{\mathcal{E}}_{a,j}^P)}.
\end{align*}

By~\eqref{formula_AC_tangent_cone}, the last expression is equal to $\Delta_{({\bf C},P)}(t^{m+1})$ and hence~(\ref{formula_AC_SIS}) is exactly the formula of the statement.
\end{proof}
\begin{remark}
Note that the first part of $\Delta(t)$ is closely related to the zeta function of the tangent cone $f_m(x,y,z)$ regarded as an function on~$\C^3$. In fact,
$
  Z(f_m: \C^3 \to \C;\, t) = (1-t^m)^{ \chi(\P^2 \setminus {\bf C}) }.
$
This is a consequence of the fact that the monodromy zeta function of a homogeneous polynomial of degree $d$ is $Z(t) = (1-t^d)^{\chi(F)/d}$, where~$F$ is the corresponding Milnor fiber.
\end{remark}

\begin{cor}\label{cor_char_poly_SIS}
The Milnor number of a SIS can be expressed in terms of the Milnor numbers of the singular points of the tangent cone, namely
$$
\mu(V,0) = (m-1)^3 + \sum_{P\in \Sing ({\bf C})} \mu ({\bf C},P).
$$
\end{cor}

\begin{proof}
The Milnor number coincides with the degree of the characteristic polynomial. Then,
\begin{align*}
\deg \Delta(t) &= m \big(m^2 - 3m + 3 - \sum_P \mu_P \big) - 1 + \sum_P \underbrace{\deg \Delta_P(t)}_{\mu_P} (m+1) \\
&= m^3 - 3 m^2 + 3m - m \sum_P \mu_P -1 + (m+1) \sum_P \mu_P \\
& = (m-1)^3 + \sum_P \mu_P. 
\end{align*}
Above, the sums are taken over $P \in \Sing({\bf C})$.
\end{proof}

\section{Yomdin-L\^{e} Surface Singularities}\label{sec_preparations_YS}

The family of singularities studied above can be generalized as follows. Let $f = f_{m} + f_{m+k} + \cdots \in \C \{x,y,z\}$ be the decomposition of $f$ into its homogeneous parts, $k \geq 1$. Denote $V := V(f) \subset \C^3$ and ${\bf C} := V(f_m) \subset \P^2$. Then, the germ $(V,0)$ is said to be a {\em Yomdin-L\^{e} surface singularity} (YLS) if the condition $\Sing({\bf C}) \cap V(f_{m+k}) = \emptyset$ holds in $\P^2$.

\vspace{0.15cm}

The main difficulty in finding a (usual) embedded resolution of this kind of singularities is that after several blow-ups at points and rational curves, following the ideas of~\cite{Artal94}, one eventually obtains a branch of resolutions depending on $k$. Thus the study of this singularities by using these tools seem to be very long and tedious.

\vspace{0.15cm}

However, an embedded $\Q$-resolution of $(V,0)$ can be computed exactly as for SIS, i.e.~by means of weighted blow-ups at points. In fact, this is the main purpose of Section~\ref{complete_description_YS}. As an application, the characteristic polynomial and the Milnor number are calculated using Theorem~\ref{ATH2}. Again, the weights at each step can be chosen so that every exceptional divisor in the $\Q$-resolution, except perhaps the first one $E_0$, contributes to its monodromy. 

\vspace{0.15cm}

In order not to repeat the same arguments, the proofs of this section are sketched, commented, or simply omitted. Moreover, they are presented following the same structure as in previous sections so that one can easily compares the corresponding results with the SIS. In the discussion, one usual thinks that $k \neq 1$, since otherwise $(V,0)$ is a SIS.

\vspace{0.15cm}

We start the embedded $\Q$-resolution of $(V,0)$ with the usual blow-up at the origin $\pi_0: \widehat{\C}^3 \to \C^3$. The total transform is the divisor $\pi_0^{*}(V) = \widehat{V} + m E_0$, where $\widehat{V}$ is the strict transform and $E_0$ is the exceptional divisor. The intersection $\widehat{V}\cap E_0$ is identified with the tangent cone of the singularity, see Figure~\ref{fig_step_zero_YS}.

\begin{figure}[h t]
\centering
\includegraphics{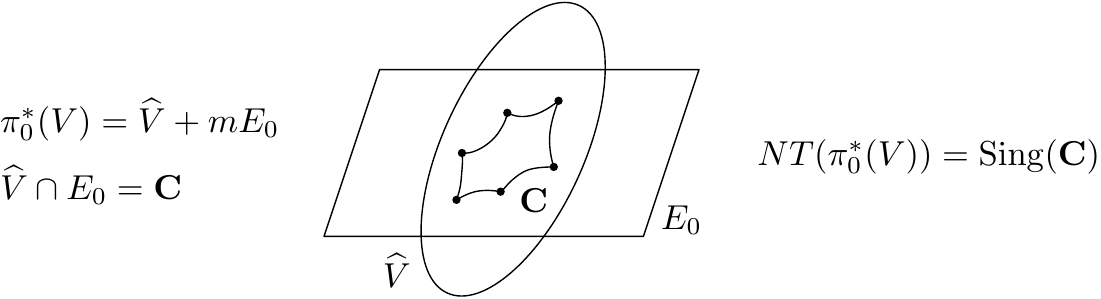}
\caption{Step 0 in the embedded $\Q$-resolution of $(V,0)$.}
\label{fig_step_zero_YS}
\end{figure}

Let us consider $P\in \widehat{V}\cap E_0 = {\bf C}$. After linear change of coordinates we can assume that $P = ((0,0,0),[0:0:1]) \equiv [0:0:1] \in {\bf C}$. Take a chart of $\widehat{\C}^3$ around $P$ where $z=0$ is the equation of $E_0$ and the blowing-up takes the form
$$
  (x,y,z) \stackrel{\pi_0}{\longmapsto} (x z, y z, z).
$$

Then, the equation of $\widehat{V}$ is
$$
  \widehat{V} :\quad f_{m}(x,y,1) + z^k \Big[ f_{m+k}(x,y,1) + z f_{m+k+1}(x,y,1) +
  \cdots \Big] = 0.
$$

Two cases arise: if $P$ is smooth in the tangent cone, then $\widehat{V}$ is also smooth at $P$ and the intersection with $E_0$ at that point is transverse; otherwise, i.e.~$P\in \Sing({\bf C})$, the YLS condition $\Sing({\bf C})\cap V(f_{m+k}) = \emptyset$ implies that the previous expression in brackets is a unit in the local ring $\C \{x,y,z \}$ and $\widehat{V}$ is not smooth at $P$ (unless $k=1$). Now the order of $f_m(x,y,1)$ is greater than or equal to 2 and the intersection $\widehat{V}\cap E_0$ is not transverse at $P$.


We summarize the previous discussion in the following result, which is the step zero in our $\Q$-resolution of $(V,0)$.

\begin{lemma}[Step 0]\label{step_zero_YS}
Let $P\in {\bf C}$. The surfaces $\widehat{V}$ and $E_0$ intersect transversely at $P$ if and only if $P$ is a smooth point in ${\bf C}$. Otherwise, i.e.~$P\in \Sing({\bf C})$, there exist local analytic coordinates around $P$ such that the equations of the exceptional divisor and the strict transform are
\begin{align*}
E_0 & : \quad z = 0\,;\\
\widehat{V} & : \quad z^k + h(x,y) = 0\,,
\end{align*}
where $h(x,y)=0$ is an equation of ${\bf C}$ and its order is at least 2.
\end{lemma}

\begin{remark}
Observe that the main difference at this stage is that $\widehat{V}$ is not smooth at the singular points of the tangent cone and its equation at those points has $z^k$ as one of its terms.
\end{remark}

\section{Embedded \textbf{Q}-Resolution for YLS}\label{complete_description_YS}

After the step zero $NT(\pi_0^{*} (V))$ is identified with $\Sing({\bf C})$. The next step in the $\Q$-resolution of $(V,0)$ is to blow up those points. Let us fix $P \in \Sing({\bf C})$ and consider local coordinates as in Lemma~\ref{step_zero_YS}. The idea is to choose suitable weights so that the strict transform of $\widehat{V}$ has again an equation of the same form, namely $z^k + H(x,y) = 0$.


Given an exceptional divisor in the tangent cone $\E_a$, $a \in S(\Gamma_{+})$, and $m_a$ its multiplicity, denote $k_a := \gcd(k,m_a)$. When $a=1$, then $m_1 = \nu_1$ and thus $k_1 = \gcd(k,\nu_1)$.

\begin{lemma}[Step 1]\label{step1_YS}
Let $(p_1,q_1)\in \N^2$ be two positive coprime numbers. Let $\varpi_1$ be the $(p_1,q_1)$-weighted blow-up at $P\in {\bf C}$. Denote by $\E_1$ its exceptional divisor and by $\nu_1$ the $(p_1,q_1)$-multiplicity of ${\bf C}$ at~$P$.

Consider $\pi_1$ the $\big( \frac{k p_1}{k_1}, \frac{k q_1}{k_1}, \frac{\nu_1}{k_1} \big)$-weighted blow-up at $P$ in dimension~$3$ and~$E_1$ the corresponding exceptional divisor. Then, the total transforms verify:
\begin{enumerate}
\item $\varpi_1^{*} ({\bf C}) = {\bf C} + \nu_1 \E_1$,
\item $\pi_1^{*} \pi^{*}_0 (V) = \widehat{V} + m E_0 + (m+k) \displaystyle \frac{\nu_1}{k_1} E_1$,
\item $NT(\pi_1^{*}\pi^{*}_0(V)) = NT(\varpi_1^{*}({\bf C}))$.
\end{enumerate}
\end{lemma}

\begin{proof}
The weighted blow-up at $P$ in the tangent cone is described in detail in the first part of the proof of Lemma~\ref{step1}. Thus we only consider here the weighted blow-up at $P$ with respect to $\big( \frac{k p_1}{k_1}, \frac{k q_1}{k_1}, \frac{\nu_1}{k_1} \big)$ in dimension $3$. 

The new space has in general three cyclic quotient singular lines, see Remark~\ref{differ2}(1) below, each of them isomorphic to~$\P^1$, and located at the new exceptional divisor $E_1$. They correspond to the three lines at infinity of~$E_1 = \P^2 \big( \frac{k p_1}{k_1}, \frac{k q_1}{k_1}, \frac{\nu_1}{k_1} \big)$.

The multiplicity of $E_1$ is the sum of the multiplicities, in our local coordinates, of the components of the divisor $\pi^{*}_0(V)$ that pass through~$P$, that is, $m \frac{\nu_1}{k_1} + k \frac{\nu_1}{k_1} = (m+k)\frac{\nu_1}{k_1}$.

Hence the total transform is the divisor
$$
\pi_1^{*} \pi^{*}_0(V) = \widehat{V} + m E_0 + (m+k) \frac{\nu_1}{k_1} E_1.
$$

To study the locus of non-transversality, the equations in the three charts are calculated in the table below. Note that the cyclic quotient spaces are represented by their normalized types, since $\gcd \big(\frac{k p_1}{k_1}, \frac{k q_1}{k_1}, \frac{\nu_1}{k_1} \big)=1$, see Example~\ref{blowup_dim3_smooth}.
$$
\begin{array}{|c|c c c|c c|}
\hline &&&&& \\[-0.325cm]
&& X \left( \displaystyle \frac{k p_1}{k_1}; -1, \frac{k q_1}{k_1}, \frac{\nu_1}{k_1} \right) &&& X \left( \displaystyle \frac{k q_1}{k_1}; \frac{k p_1}{k_1}, -1, \frac{\nu_1}{k_1} \right) \\[0.4cm]
( x,y,z ) \stackrel{\pi_1}{\longmapsto} && ( x^{\frac{k p_1}{k_1}}, x^{\frac{k q_1}{k_1}} y, x^{\frac{\nu_1}{k_1}} z ) &&& (x y^{\frac{k p_1}{k_1}}, y^{\frac{k q_1}{k_1}}, y^{\frac{\nu_1}{k_1}} z ) \\[0.1cm]
\hline
E_0 && z=0 &&& z=0 \\
E_1 && x=0 &&& y=0 \\
\widehat{V} && z^k+h_1(x^{\frac{k}{k_1}},y)=0 &&& z^k + h_2(x,y^{\frac{k}{k_1}})=0 \\
\hline
\end{array}
$$

\vspace{0.1cm}

$$
\begin{array}{|c|c c|}
\hline && \\[-0.325cm]
&& X \left( \displaystyle\frac{\nu_1}{k_1}; \frac{k p_1}{k_1}, \frac{k q_1}{k_1},-1 \right)\\[0.4cm]
(x,y,z) \stackrel{\pi_1}{\longmapsto} && (x z^{\frac{k p_1}{k_1}}, y z^{\frac{k q_1}{k_1}}, z^{\frac{\nu_1}{k_1}}) \\[0.1cm]
\hline
E_0 && -\\
E_1 && z=0\\
\widehat{V} && 1 + h_{\nu_1}(x,y) + z^{\frac{kl}{k_1}} h_{\nu_1+l}(x,y) + \cdots = 0\\
\hline
\end{array}
$$

\vspace{0.25cm}

Clearly $E_1$ and $E_0$ intersect transversely. The strict transform $\widehat{V}$ also cuts $E_1$ transversely except perhaps at $\{ z = 0 \} \subset E_1$. The equations of these intersections are given by
\begin{align*}
E_0\cap E_1 & = \{z=0\}, \\
\widehat{V}\cap E_1 & = \{ z^k + h_{\nu_1}(x,y) = 0\},
\end{align*}
as projective subvarieties in $E_1 = \P^2 \big(\frac{k p_1}{k_1}, \frac{k q_1}{k_1}, \frac{\nu_1}{k_1} \big)$.

By Proposition~\ref{bezout_th_P2w-mu_d}, these smooth projective curves have self-intersection numbers $\frac{k_1 \nu_1}{k^2 p_1 q_1}$ and $\frac{k_1 \nu_1}{p_1 q_1}$ respectively. They meet at $\#({\bf C}\cap \E_1)$ points with intersection number $k_1/k$ times the intersection number in ${\bf C}\cap \E_1$, that is, for $P \in {\bf C} \cap \mathcal{E}_1 \equiv \widehat{V} \cap E_0 \cap E_1$, one has
\begin{equation}\label{inter_YS}
\left( \widehat{V} \cap E_1, E_0 \cap E_1;\, E_1 \right)_P = \frac{k_1}{k} \cdot \left( {\bf C}, \mathcal{E}_1;\, \widehat{\C}^2_{(p_1,q_1)} \right)_P.
\end{equation}

On the other hand, the intersection of the total transform with $E_0$ produces an identical situation to the tangent cone, see Remark~\ref{differ2}(2) for a more detailed explanation.

\begin{figure}[h t]
\centering
\includegraphics[scale=0.94]{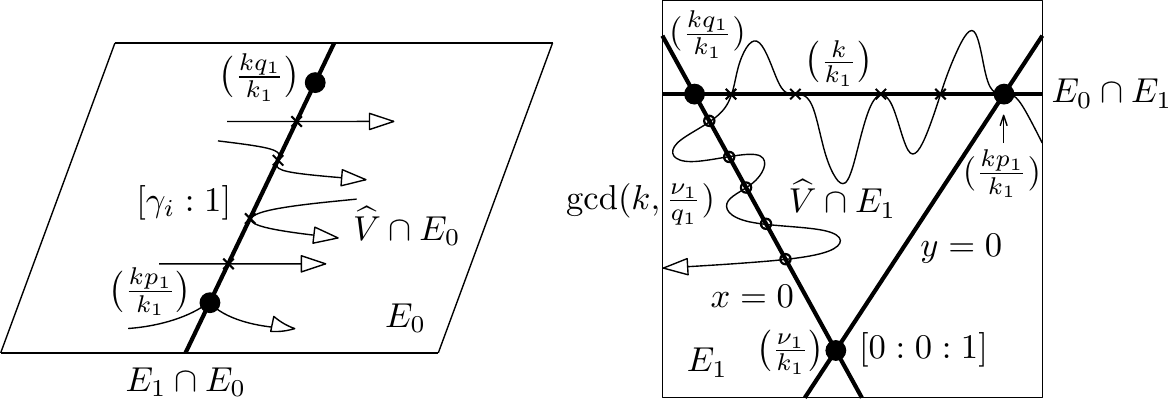}
\caption{Step $1$ in the embedded ${\bf Q}$-resolution of $(V,0)$.}\label{step1_V_YS}
\end{figure}

All these statements follow from the equations above. In Figure~\ref{step1_V_YS}, we see the intersection of the divisor $\pi^{*}_1 \pi^{*}_0 (V)$ with $E_0$ and~$E_1$, respectively. See also Figure~\ref{step1_C} for the situation in ${\bf C}$.
Finally, the triple points of the total transform in dimension $3$ are identified with the points of ${\bf C}\cap \E_1$ and, by~(\ref{inter_YS}), the intersection at one of those points is transverse if and only if so is it in dimension $2$. This concludes the proof.
\end{proof}

\begin{remark}\label{differ2}
Just to emphasize, we collect below the main differences with the embedded $\Q$-resolution for SIS at this stage, cf.~Lemma~\ref{step1} and its proof.
\begin{enumerate}
\item The stratum $\{ z = 0 \} \setminus \{ [0:1:0], [1:0:0] \} \subset E_1$ contains singular points of the ambient space. In fact, the group acting on these points is of type $\big( \frac{k}{k_1}; -1, 0, \frac{\nu_1}{k_1} \big)$, see Figure~\ref{step1_V_YS}.
\item In principle, the intersection of $E_0$ with the rest of components seem to be different from the situation in the tangent cone, because in the first chart $E_1 \cap E_0 = \{ x=0 \}$ and $\widehat{V} \cap E_0 = \{ h_1(x^{k/k_1},y) = 0 \}$ on $X\big( \frac{k p_1}{k_1}; -1, \frac{k q_1}{k_1} \big)$. After normalizing the latter type, one finds the equation of $\E_1$ and ${\bf C}$ on $X(p_1;-1,q_1)$, cf.~\ref{justification_non-normal_YS}.
\item Write $h_{\nu_1}(x,y) = x^{a} y^{b} \prod_{i} (x^{q_1} - \gamma_i^{q_1} y^{p_1})^{e_i} = 0$. If $a=0$, or equivalently $\E_1 \ni [0:1] \notin {\bf C}$, then $\{ x= 0\} \subset E_1$ cuts $\widehat{V} \cap E_1 = \{ z^k + h_{\nu_1}(x,y) = 0 \}$ in exactly $\gcd(k, \frac{\nu_1}{q_1})$ points different from the origins of $E_1$. Analogously, $\{ y = 0 \} \subset E_1$ intersects in $\gcd(k,\frac{\nu_1}{p_1})$ points if $b=0$. This can be checked directly or applying Bézout's Theorem on $E_1$.
\end{enumerate}
\end{remark}


Let $\Gamma$ and $\Gamma_{+}$ be the dual graphs associated with the total transform and the exceptional divisor, after having computed an embedded $\Q$-resolution of $({\bf C}, P)$, respectively. Denote by $S(\Gamma)$ and $S(\Gamma_{+})$ the sets of their vertices. The classical partial order on $S(\Gamma_{+})$ is denoted by $\preccurlyeq$.

\vspace{0.15cm}

The locus of non-transversality after the last blow-up in dimension~3 is identified with the locus of non-transversality in the resolution of~$({\bf C},P)$. Each of these points corresponds to a weighted blow-up in the resolution of the tangent cone, that is, to a vertex of $\Gamma_{+}$. Thus in the next step we need to blow-up those points to produce a similar situation. Again the same operation will be applied to the points where the total transform is not a normal crossing divisor. These points will be associated with vertices of $\Gamma_{+}$ too.

\vspace{0.15cm}

Before describing a generic step, blowing up the point $P_b$ as in Proposition~\ref{step_b_SIS}, let us clarify the justification for working with non-normalized spaces.

\begin{nothing}\label{justification_non-normal_YS}
After the first blow-up the local equation of the total transform of $({\bf C},P)$ is given by $x^{\nu_1} h_1(x,y) : X(p_1; -1, q_1) \to \C$, see proof of Lemma~\ref{step1}. The situation in dimension $3$ is provided by
$$
  \underbrace{x^{(m+k) \frac{\nu_1}{k_1}}_{}}_{E_1} \cdot \underbrace{z^m_{}}_{E_0} \cdot \underbrace{\big[ z^k + h_1(x^{\frac{k}{k_1}},y) \big]}_{\widehat{V}} :\, X \bigg( \frac{k p_1}{k_1}; -1, \frac{k q_1}{k_1}, \frac{\nu_1}{k_1} \bigg) \longrightarrow \C,
$$
as we have just seen in the proof of Lemma~\ref{step1_YS}. The divisors $E_1$ and $\E_1$ are both represented by $x=0$.

However, the equation of the strict transform of ${\bf C}$ and $\widehat{V}$ do not correspond to each other directly. This obstruction can be solved working with non-normalized types, since the function
$$
  x^{\frac{k \nu_1}{k_1}} h_1 (x^{\frac{k}{k_1}},y):\, X \bigg( \frac{k p_1}{k_1}; -1, \frac{k q_1}{k_1} \bigg) \longrightarrow \C
$$
also gives rise to the total transform of ${\bf C}$ on a space represented by a non-normalized type.

On the other hand, the embedded $\Q$-resolution of a Yomdin-L\^{e} surface singularity will contain in general non-cyclic quotient singularities. Hence providing normalized types is long and tedious. Motivated by this fact and for better understanding of the relationship  between ${\bf C}$ and $(V,0)$, we present the embedded $\Q$-resolution without explicitly giving the normalized type of each quotient space.
\end{nothing}


The following result is proven by induction on $S(\Gamma_{+})$ using the relation~$\preccurlyeq$. Lemma \ref{step1_YS} and~\ref{justification_non-normal_YS} just above is the first step in the induction.
Let $b\in S(\Gamma_{+})$ be a vertex such that $P_b$ belongs to the locus of non-transversality of the total transform. As usual, denote by $\E_b$ the exceptional divisor appearing after blowing up the point $P_b$.

\begin{prop}[Step $b$]\label{step_b_YS}
Let $\varpi_b$ be the $(p_b,q_b)$-weighted blow-up at $P_b$ with $b\in S(\Gamma_{+})$. Denote by $\E_b$ its exceptional divisor, $\nu_b$ the $(p_b,q_b)$-multiplicity of ${\bf C} \subset \C^2$, and $m_b$ the multiplicity of $\E_b$. Assume, if necessary, that $k|p_b$ and $k|q_b$ so that $k|\nu_b$ too.

Consider $\pi_b$ the $(p_b,q_b,\frac{\nu_b}{k})$-weighted blow-up at $P_b$ in dimension $3$ and $E_b$ the corresponding exceptional divisor. Then, after blowing up the point~$P_b$, the new total transform verifies: 
\begin{enumerate}
\item The exceptional divisor $E_b$ is isomorphic to $\P^2 (p_b, q_b, \frac{\nu_b}{k}) / \mu_{\be}$ and its multiplicity equals $(m+k) \frac{m_b}{k_b}$. In general, their three lines at infinity are quotient singular in the ambient space.

\item Let $a$ be a vertex such that $a \prec b$. Then, $E_a \cap E_b \neq \emptyset$ if and only if
$P_b \in \E_a$. In such a case, the curve $E_a \cap E_b$ is one of the two lines at infinity of $E_b$ different from~$\{z=0\}$. If $P_b \in \E_a \cap \E_{a'}$, $a\neq a'$, then the corresponding lines are different and hence they meet at the point $[0:0:1]$.

\item The intersection of the rest of components with $E_0$ produces an identical situation to the resolution of $({\bf C},P)$, after blowing up the point $P_b$. More precisely,
\begin{align*}
\widehat{V} \cap E_0 & = {\bf C}, \\
E_b \cap E_0 & = \E_b, \\
E_a \cap E_0 & = \E_a, \quad \forall a \preccurlyeq b.
\end{align*}

\item The curves $E_0 \cap E_b = \{z=0\}$ and $\widehat{V} \cap E_b = \{ z^{k} + H_{\nu_b}(x,y) = 0 \}$ have self-intersection numbers $\frac{- \E_b^2 \nu_b k_b}{k^2 \l}$ and $\frac{- \E_b^2 \nu_b k_b}{\l}$ respectively, and the intersecting points can be identified with ${\bf C}\cap \E_b$.

Moreover, the intersection multiplicity of these two curve at those points can be computed as follows. Let $P \in \widehat{V} \cap E_0 \cap E_b \equiv {\bf C} \cap \mathcal{E}_b$, then one has
$$
\hspace{1.5cm} \left( \widehat{V} \cap E_b, E_0 \cap E_b;\, E_b \right)_P = \frac{1}{O(E_{b,z})} \cdot \left( {\bf C}, \mathcal{E}_b;\, \widehat{\C}^2_{(p_b, q_b)} \big/ \mu_{e} \right)_P,
$$
where $O(E_{b,z})$ denotes the order of the group acting on the natural stratum $E_{b,z} := \{ z = 0 \} \setminus \{ [0:1:0], [1:0:0] \} \subset E_b$.

Let $P_b\in \E_a$ ($a  \prec b$) and assume e.g.~$E_a \cap E_b = \{ x= 0 \} \subset E_b$. If ${\bf C} \cap \E_a \cap E_b = \emptyset$, then $E_a \cap E_b$ and $\widehat{V}\cap E_b$ meet transversely at exactly $\gcd(k,m(\check{\E}_{b,x}))$ points different from the origins of $E_b$. Otherwise, i.e.~${\bf C} \cap \E_a \cap E_b \neq \emptyset$, the letter curves only meet at one point, which besides passes through $E_0 \cap E_b$. This is the case when there exist quadruple points.

\item The locus of non-transversality of the total transform in dimension~3 is identified with the one in the resolution of $({\bf C},P)$. These points belong to $\widehat{V} \cap E_0 \cap E_b \equiv {\bf C} \cap \E_b$ and they correspond to the ones where the curves $E_0 \cap E_b$ and $\widehat{V} \cap E_b$, or equivalently $\E_b$ and~${\bf C}$, do not meet transversely.

\item The strict transform $\widehat{V}$ never passes through $[0:0:1] \in E_b$.
\end{enumerate}
\end{prop}

\begin{proof}
By induction on $S(\Gamma_{+})$ with respect to the order $\preccurlyeq$. The base case is Lemma~\ref{step1_YS} together with its modification explained in~\ref{justification_non-normal_YS}. As for the inductive step, one proceeds as in the proof of Lemma~\ref{step2}. Assume, by induction, that the local equation of the total transform in the resolution of the tangent cone around~$P_b$ is given by the function
\begin{equation*}
  x^{n_{a}} y^{n_{a'}} H(x,y)\, :\, X(\be; \br, \bs) \longrightarrow \C,
\end{equation*}
where ${\bf C} = \{ H(x,y) = 0 \}$ is the equation of the strict transform and the others correspond to the divisors $\E_a$ and $\E_{a'}$ (they may not appear if $n_a$ or $n_{a'}$ equals zero). In principle, the type $(\be; \br, \bs)$ is not assumed to be normalized. Hence $n_a$ and $n_{a'}$ are not the multiplicities of $\E_{a}$ and $\E_{a'}$.

Also, the equation of the total transform around $P_b$ in dimension~$3$ is given by the function
\begin{equation*}
x^{\frac{(m+k) n_{a}}{k}} \cdot y^{\frac{(m+k) n_{a'}}{k}} \cdot z^{m} \cdot \big[ z^k + H(x,y) \big]\, :\, X(\be; \br, \bs, \bt) \longrightarrow \C,
\end{equation*}
where $\widehat{V} = \{ z^k + H(x,y) = 0\}$ is the strict transform, $E_0 = \{z=0\}$, and the others are the divisors $E_a$ and $E_{a'}$ (if they exist). Using that both equations are well-defined functions on the corresponding quotient spaces, one has
\begin{equation}\label{condition_on_t}
\frac{n_{a}}{k} \cdot \br + \frac{n_{a'}}{k} \cdot \bs + \bt \equiv 0 \quad (\text{mod}\ \be).
\end{equation}

The verification of the statement is very simple once the local equations of the divisors appearing in the total transform are calculated. The main ideas behind are contained in the proof of Lemma~\ref{step1_YS} and~\ref{justification_non-normal_YS}. The details are omitted to avoid repeating the same arguments; only the local equations are given, see below. To do so, consider the following data and use the charts described in Examples~\ref{blowup_dim2} and~\ref{blowup_dim3_singular}. As auxiliary results Propositions~\ref{formula_self-intersection} and~\ref{bezout_th_P2w-mu_d} and Remark~\ref{computation_local_number_V-surface} are also needed.
$$
\begin{array}{lcl}
\nu_b := \ord_{(p_b,q_b)} H (x,y) & \hspace{0.75cm} & n_b := p_b \cdot n_a + q_b \cdot n_{a'} + \nu_b \\[0.2cm]
H_1(x,y) := \frac{H(x^{p_b}, x^{q_b} y)}{x^{\nu_b}} && H_2(x,y) := \frac{H(x y^{p_b},y^{q_b})}{y^{\nu_b}}
\end{array}
$$

Note that if $Q_1^{\bf C}$ denotes the quotient space of the first chart in the tangent cone (see below) and $(Q_1^{\bf C}, [(0,1)]) \cong (\C^2, (0,1))$, $[(x,y)] \mapsto (x^{\l},y)$ defines an isomorphism of germs, then the multiplicity of the new exceptional divisor $\E_b$ is $m_b = \frac{n_b}{\l}$.

These are the equations in the resolution of the tangent cone. They are presented as zero sets omitting their multiplicities.

\vspace{0.1cm}

\begin{center}
\begin{tabular}{|ll|c|}
\hline
\multicolumn{2}{|c|}{Equations} & Chart\\[0.175cm]
\hline
$\E_b:$ & $x=0$ & \multirow{3}{*}{$\displaystyle X \bigg( \begin{array}{c | c c} p_b & -1 & q_b \\ p_b \be & \br & p_b \bs - q_b \br \end{array} \bigg) \ \longrightarrow \ \, \widehat{\C}^2(p_b,q_b)\Big/\mu_{\be}$}\\
$\E_{a}:$ & $-$ &\\
$\E_{a'}$ & $y=0$ & \\
${\bf C}:$ & $H_1 (x,y) = 0$ & $\big[(x,y)\big] \ \mapsto \ \big[ \big((x^{p_b},x^{q_b} y),[1:y]_{(p_b,q_b)}\big) \big]$\\[0.25cm]
\hline 
$\E_b:$ & $y=0$ & \multirow{3}{*}{$\displaystyle X \bigg( \begin{array}{c | c c} q_b & p_b & -1 \\ q_b \be & q_b \br - p_b  \bs & \bs \end{array} \bigg) \ \longrightarrow \ \, \widehat{\C}^2(p_b,q_b)\Big/\mu_{\be}$}\\
$\E_a:$ & $x=0$ &\\
$\E_{a'}:$ & $-$ & \\
${\bf C}:$ & $H_2 (x,y) = 0$ & $\big[(x,y)\big] \ \mapsto \ \big[ \big((x y^{p_b},y^{q_b}),[x:1]_{(p_b,q_b)}\big) \big]$\\[0.25cm]
\hline
\end{tabular}
\end{center}

\vspace{0.1cm}

In dimension~$3$, the local equations of the total transform are presented as well-defined functions over the corresponding quotient spaces. The notation is self-explanatory to recognize the equation of each divisor. In the first chart, however, it is indicated the divisor corresponding to each equation. Note that, for instance, the polynomial in the first chart has been obtained after performing the substitution $(x,y,z) \mapsto (x^{p_b}, x^{q_b} y, x^{\frac{\nu_b}{k}} z)$. \vspace{0.1cm}
$$
\begin{array}{r|cl}
\text{1st chart\ } && X \bigg( \begin{array}{c | c c c} p_b & -1 & q_b & \frac{\nu_b}{k} \\ p_b \be & \br & p_b \bs - q_b \br & p_b \bt - \frac{\nu_b}{k} \br \end{array} \bigg) \longrightarrow \C \\[0.5cm]
&& \underbrace{x^{\frac{(m+k) n_b}{k}}_{}}_{E_b} \cdot \underbrace{y^{\frac{(m+k) n_{a'}}{k}}_{}}_{E_{a'}} \cdot \underbrace{z^m_{}}_{E_0} \cdot \underbrace{\big[ z^k + H_1 (x, y) \big]}_{\widehat{V}}
\\ \multicolumn{2}{c}{} \\
\text{2nd chart\ } && X \bigg( \begin{array}{c | c c c} q_b & p_b & -1 & \frac{\nu_b}{k} \\ q_b \be & q_b \br - p_b  \bs & \bs & q_b \bt - \frac{\nu_b}{k} \bs \end{array} \bigg) \longrightarrow \C \\[0.5cm]
&& x^{\frac{(m+k) n_a}{k}}\cdot y^{\frac{(m+k) n_b}{k}} \cdot z^m \cdot \big[ z^k + H_2 (x, y) \big]
\\ \multicolumn{2}{c}{} \\
\text{3rd chart\ } && X \bigg( \begin{array}{c | c c c} \frac{\nu_b}{k} & p_b & q_b & -1 \\ \frac{\nu_b}{k} \be & \frac{\nu_b}{k} \br - p_b \bt & \frac{\nu_b}{k} \bs - q_b \bt & \bt \end{array} \bigg) \longrightarrow \C \\[0.5cm]
&& x^{\frac{(m+k) n_a}{k}} \cdot y^{\frac{(m+k) n_{a'}}{k}} \cdot z^{\frac{(m+k) n_b}{k}} \cdot \Big[ 1 + \frac{H(x z^{p_b}, y z^{q_b})}{z^{\nu_b}} \Big]
\end{array}
$$

\vspace{0.1cm}

Note that if $Q_1^{\bf V}$ denotes the quotient space of the first chart in dimension~3 (see above) and $(Q_1^{\bf V}, [(0,1,1)]) \cong (\C^3, (0,1,1))$, $[(x,y,z)] \mapsto (x^{L},y,z)$ defines an isomorphism of germs, then the multiplicity of the new exceptional divisor $E_b$ is $\frac{(m+k) n_b}{k L}$.
\end{proof}

\begin{remark}
Observe that the columns of the new spaces satisfy a condition analogous to~\eqref{condition_on_t}. For example, using~\eqref{condition_on_t}, it can be checked that
$$
  \frac{n_b}{k} \cdot \begin{pmatrix} -1 \\ \br \end{pmatrix} + \frac{n_{a'}}{k} \cdot \begin{pmatrix} q_b \\ p_b \bs - q_b \br \end{pmatrix} + \begin{pmatrix} \frac{\nu_b}{k} \\ p_b \bt - \frac{\nu_b}{k} \br \end{pmatrix} \equiv \begin{pmatrix} 0 \\ \mathbf{0} \end{pmatrix}, \quad \text{mod} \begin{pmatrix} p_b \\ p_b \be \end{pmatrix}.
$$
In other words, the third column is a linear combination of the first two ones, modulo the order of the corresponding group. This can be used to prove that $L = \gcd(\l, \frac{n_b}{k})$ and hence the multiplicity of $E_b$ is $\frac{(m+k) \cdot m_b}{\gcd(k,m_b)}$ indeed.
\end{remark}

\begin{theo}\label{summery_theorem_YS}
Given an embedded $\Q$-resolution of $({\bf C},P)$ for all $\, P\in \Sing({\bf C})$, one can construct an embedded $\Q$-resolution of $(V,0)$, consisting of weighted blow-ups at points. Each of these blow-ups corresponds to a weighted blow-up in the resolution of $({\bf C},P)$ for some $P\in \Sing ({\bf C})$, that is, it corresponds to a vertex of $\Gamma_{+}^P$. \hfill $\Box$
\end{theo}

By Lemma~\ref{euler_char_YS} and Theorem~\ref{char_poly_YS}, an exceptional divisor $E_a^P$ in the $\Q$-resolution of~$(V,0)$ contributes to the monodromy if and only if so does the corresponding divisor $\E_a^P$ in $({\bf C},P)$. Hence the weights can be chosen so that every exceptional divisor, except perhaps the first one $E_0$, contributes to its monodromy.

\section{The Characteristic Polynomial of YLS}\label{sec_char_poly_YS}

Here we plan to apply Theorem~\ref{ATH2} to compute the characteristic polynomial of the monodromy and the Milnor number of $(V,0)$ in terms of its tangent cone $({\bf C},P)$. Some notation need to be introduced, concerning the stratification of each irreducible component of the exceptional divisor in terms of its quotient singularities.

\vspace{0.25cm}

Given a point $P \in \Sing({\bf C})$, denote by $\varrho^{P}: Y^p \to ({\bf C},P)$ an embedded $\Q$-resolution of the tangent cone. Assume that the total transform is given by
$$
(\varrho^{P})^{*}({\bf C},P) = {\bf C} + \sum_{a \in S(\Gamma_{+}^{P})} m_a^{P} \mathcal{E}_a^{P},
$$
where $\E_a^{P}$ is the exceptional divisor of the $(p_a^P,q_a^P)$-blow-up at a point $P_a$ belonging to the locus of non-transversality. Denote by $\nu_a^P$ the $(p_a^P,q_a^P)$-multiplicity of ${\bf C}$ at $P_a$.

Recall that $\E_a^P$ is naturally isomorphic to $\P^1_{(p_a^P,q_a^P)} / \mu_{\be}$. Using this identification, see Figure~\ref{strata_EaP_YS}, define
$$
\E^P_{a,1} = \E^{P}_a \setminus \{ [0:1], [1:0] \}, \qquad
\E^P_{a,x} = \{ [0:1] \}, \qquad
\E^P_{a,y} = \{ [1:0] \}.
$$
The strata $\check{\E}_{a,j}^P := \E^{P}_{a,j} \setminus \big( \E_{a,j}^{P} \cap \big( \bigcup_{b \neq a} \E_b^P \cup {\bf C} \big) \big)$ for $j=1,x,y$ (see notation just above Theorem~\ref{ATH2}) will be considered in Lemma~\ref{euler_char_YS}.

\vspace{0.25cm}

Let us see the situation in the Yomdin-L\^{e} singularity $(V,0)$. Denote by $\rho: X \to (V,0)$ the embedded $\Q$-resolution obtained following Proposition~\ref{step_b_YS}. Then, the total transform is (recall $k_a^P := \gcd(k,m_a^P)$)
$$
  \rho^{*}(V,0) = \widehat{V} + m E_0 + \sum_{\begin{subarray}{c} P \in \Sing({\bf C}) \\ a \in S(\Gamma_{+}^{P})\end{subarray}} (m+k) \frac{m_a^P}{k_a^P} E_{a}^{P},
$$
and $E_a^P$ appears after the blow-up at the point $P_a$ with suitable weights (recall that the locus of non-transversality in dimension 2 and 3 are identified).

The divisor $E_a^P$ is naturally isomorphic to $\P^2_{\w} / \mu_{\be}$. Using this identification, see Figure~\ref{strata_EaP_YS}, define
$$
\begin{array}{l}
E_{a,1}^P = E_a^P \setminus \{ xyz = 0 \}, \qquad E_{a,x}^P = \{ x=0 \} \setminus \{ [0:1:0], [0:0:1] \}, \\[0.2cm]
E_{a,y}^P = \{ y=0 \} \setminus \{ [1:0:0], [0:0:1] \},  \hfill E_{a,xy}^P = \{ [0:0:1 \}.
\end{array}
$$
Analogously, one considers $E_{a,z}^P$, $E_{a,xz}^P$, and $E_{a,yz}^P$ so that $E_a^P = \bigsqcup_{j} E_{a,j}^{P}$ really defines a stratification. However, these three strata belong to more than one irreducible divisor in the total transform and hence they do not contribute to the characteristic polynomial.

As for $E_0$, according to its quotient singularities, no stratification need to be considered (it is smooth).

The Euler characteristic of $\check{E}_0$ and $\check{E}_{a,j}^P := E^{P}_{a,j} \setminus \big( E_{a,j}^{P} \cap \big( \bigcup_{b \neq a} E_b^P \cup \widehat{V} \big) \big)$ for $j=1,x,y,xy$ (see notation just above Theorem~\ref{ATH2}) as well as its multiplicity are calculated in Lemma~\ref{euler_char_YS}.

\begin{figure}[h t]
\centering
\includegraphics{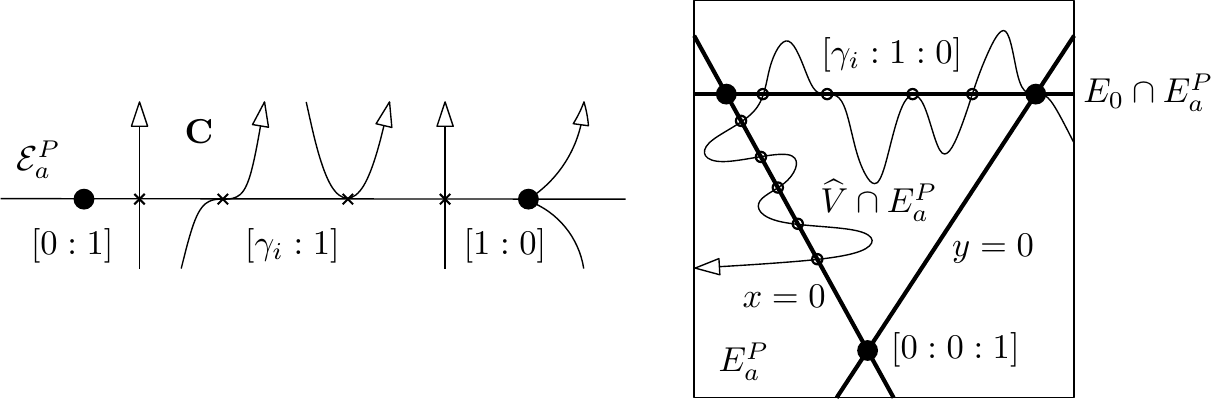}
\caption{Stratification of $\E_a^{P}$ and $E_a^{P}$.}
\label{strata_EaP_YS}
\end{figure}

The following three results are presented without their proofs because they do not provide any new idea. They are the analogous of Lemma~\ref{euler_char_SIS}, Theorem~\ref{char_poly_SIS}, and Corollary~\ref{cor_char_poly_SIS}, respectively. Anyway, recall that the Euler characteristic of $\P^2 \setminus {\bf C}$ is $m^2-3m+3-\sum_{P \in \Sing(P)} \mu_P$.

\begin{lemma}\label{euler_char_YS}
Using the previous notation, the Euler characteristic and the multiplicity of $\check{E}_0$ are
$$
\chi (\check{E}_0) = \chi (\P^2 \setminus {\bf C}), \qquad m(\check{E}_0) = m.
$$

For the rest of strata of $\check{E}_{a}^{P}$, let us fix a point $P \in \Sing({\bf C})$. Then, one has that
$$
\chi ( \check{E}_{a,j}^{P} ) = \begin{cases} 1 & a = 1,\ j=xy \\ 0 & a \neq 1,\ j = xy \\ - \gcd \big( k, m(\check{\E}_{a,j}^{P}) \big) \cdot \chi ( \check{\E}_{a,j}^{P} ) & \forall a,\ j=1,x,y\,; \end{cases}
$$
$$
\chi ( \check{\E}_{a,j}^{P} ) \neq 0 \ \Longrightarrow \ m( \check{E}_{a,j}^{P} ) = \begin{cases} m+k & a = 1,\ j=xy \\[0.25cm]
\displaystyle \frac{(m+k) \cdot m(\check{\E}_{a,j}^{P})}{\gcd \big( k, m(\check{\E}_{a,j}^{P}) \big)} & \forall a,\ j=1,x,y. \end{cases}
$$
In fact, $\forall a \in S(\Gamma_{+}^{P})$, $a \neq 1$, the stratum $\check{E}^{P}_{a,xy}$ is empty and, in particular, its Euler characteristic is zero. \hfill $\Box$ 
\end{lemma}

\begin{theo}\label{char_poly_YS}
The characteristic polynomial of the complex monodromy of $(V,0)$ is
$$
  \Delta_{(V,0)}(t) = \frac{(t^m-1)^{\chi(\P^2 \setminus {\bf C})}}{t-1}
  \prod_{P\in \Sing({\bf C})} \Delta_{({\bf C},P)}^{k} (t^{m+k}),
$$
where $\Delta_{({\bf C},P)}(t)$ denotes the characteristic polynomial of the local complex monodromy of $({\bf C},P)$ and if $\Delta(t) = \prod_i (t^{m_i} - 1)^{a_i}$, then $\Delta^k(t)$ denotes
\begin{center}
\hfill $\displaystyle \Delta^k(t) = \prod_i \left( t^{\frac{m_i}{\gcd(m_i,k)}} - 1 \right)^{\gcd(m_i,k) a_i}$. \hfill $\Box$
\end{center}
\end{theo}

\begin{cor}\label{cor_char_poly_YS}
The Milnor number of a Yomdin-L\^{e} surface singularity can be expressed in terms of the Milnor numbers of the singular points of the tangent cone, namely
\begin{center}
\hfill $\displaystyle \mu(V,0) = (m-1)^3 + k \sum_{P\in \Sing ({\bf C})} \mu ({\bf C},P)$. \hfill $\Box$
\end{center}
\end{cor}

\def\cprime{$'$}

\end{document}